\documentclass[11pt]{amsart}
\usepackage{amsxtra, amsfonts, amsmath, amsthm, amstext, amssymb, amscd, mathrsfs, verbatim, color}
\usepackage{threeparttable}
\usepackage{mathtools}
\usepackage[ansinew]{inputenc}\usepackage[T1]{fontenc}
\usepackage[all,cmtip]{xy}
\theoremstyle{plain}

\newtheorem{theorem}{Theorem}[section]
\newtheorem{lemma}[theorem]{Lemma}
\newtheorem{definition-theorem}[theorem]{Definition-Theorem}
\newtheorem{proposition}[theorem]{Proposition}
\newtheorem{corollary}[theorem]{Corollary}

\theoremstyle{definition}
\newtheorem{definition}[theorem]{Definition}
\newtheorem{example}[theorem]{Example}
\newtheorem{remark}[theorem]{Remark}
\newtheorem{notation}[theorem]{Notation}
\newcommand \bth[1] { \begin{theorem}\label{t#1} }
\newcommand \ble[1] { \begin{lemma}\label{l#1} }

\newcommand \bpr[1] { \begin{proposition}\label{p#1} }
\newcommand \bco[1] { \begin{corollary}\label{c#1} }
\newcommand \bde[1] { \begin{definition}\label{d#1}\rm }
\newcommand \bex[1] { \begin{example}\label{e#1}\rm }
\newcommand \bre[1] { \begin{remark}\label{r#1}\rm }

\newcommand \bnota[1] {\begin{notation}\label{n#1}\rm }
\newcommand {\ele} { \end{lemma} }

\newcommand {\epr} { \end{proposition} }
\newcommand {\eco} { \end{corollary} }
\newcommand {\ede} { \end{definition} }
\newcommand {\eex} { \end{example} }
\newcommand {\ere} { \end{remark} }
\newcommand {\enota} { \end{notation} }

\begin{document}
\title[Branching for standard modules]{Ext-Multiplicity Theorem for Standard Representations of $(\mathrm{GL}_{n+1},\mathrm{GL}_n)$ } 

\author[Kei Yuen Chan]{Kei Yuen Chan}
\address{Shanghai Center for Mathematical Sciences, Fudan University \\}
\email{kychan1@hku.hk}
\maketitle

\begin{abstract}
Let $\pi_1$ be a standard representation of $\mathrm{GL}_{n+1}(F)$ and let $\pi_2$ be the smooth dual of a standard representation of $\mathrm{GL}_n(F)$. When $F$ is non-Archimedean, we prove that $\mathrm{Ext}^i_{\mathrm{GL}_n(F)}(\pi_1, \pi_2)$ is $\cong \mathbb C$ when $i=0$ and vanishes when $i \geq 1$. The main tool of the proof is a notion of left and right Bernstein-Zelevinsky filtrations. An immediate consequence of the result is to give a new proof on the multiplicity at most one theorem. Along the way, we also study an application of an Euler-Poincar\'e pairing formula of D. Prasad on the coefficients of Kazhdan-Lusztig polynomials.

When $F$ is an Archimedean field, we use the left-right Bruhat-filtration to prove a multiplicity result for the equal rank Fourier-Jacobi models of standard principal series.

\end{abstract}

\section{Introduction}

Let $F$ be a local field. One of the important results in quotient branching law is the multiplicity one phenomenon: for an irreducible representation $\pi_1$ of $\mathrm{GL}_{n+1}(F)$ and $\pi_2$ of $\mathrm{GL}_n(F)$,
\[ \mathrm{dim}~\mathrm{Hom}_{\mathrm{GL}_n(F)}(\pi_1, \pi_2) \leq 1,
\]
which is a part of the Gan-Gross-Prasad problems \cite{GGP12}. This is established by Aizenbud-Gourevitch-Rallis-Schiffman \cite{AGRS10} and Sun-Zhu \cite{SZ12}.  Given the uniqueness of an irreducible quotient, the existence problem remains open for general situations. Partial progress is obtained for some special cases such as generic representations (e.g. \cite{JPSS83, GGP12}) and non-tempered representations from Arthur packets (e.g. \cite{GGP20, Ch20}) in $p$-adic fields, and generic representations (e.g. \cite{Ja09}) in real or complex fields.

Assume $F$ is non-Archimedean. Let $G_n=\mathrm{GL}_n(F)$. The goal of this paper is to study the branching problem for another important class of representations--standard modules and its Ext-analog (c.f. \cite{Pr18}). Precisely, we prove that:
\begin{theorem} \label{thm multi standard}
Let $F$ be a non-archimedean local field. Let $\pi_1$ and $\pi_2$ be standard modules of $G_{n+1}$ and $G_n$ respectively. Then 
\[  \mathrm{dim}~ \mathrm{Hom}_{G_n}(\pi_1, \pi_2^{\vee}) =1
\]
and, for all $i \geq 1$,
\[ \mathrm{Ext}^i_{G_n}(\pi_1 , \pi_2^{\vee}) =0 .
\]
\end{theorem}

 From the viewpoint of period integrals, the local uniqueness property would give the uniqueness of global periods as well as guarantee Euler factorization of period. For Archimedean case, the work of Sun-Zhu \cite{SZ12} indicates such uniqueness for the Hom multiplicity result, and we establish that for the Fourier-Jacobi models of standard principal series. In our context, the regularized periods for Eisenstein series for $\mathrm{GL}_{n+1}\times \mathrm{GL}_n$ are studied in \cite{IY15} by Ichino-Yamana. 

 From the homological viewpoint, Theorem \ref{thm multi standard} generalizes the Prasad Ext-vanishing conjecture, proved in \cite{CS21}. Indeed, an analogue result in the ordinary setting also holds, that is, when both $\pi_1$ and $\pi_2$ are standard representations of $G_n$, for all $i$,
\[  \mathrm{Ext}_{G_n}^i(\pi_1, \pi_2^{\vee}) =0 
\]
if $\pi_1 \not\cong \theta(\pi_2)$ (see Section \ref{ss notations} for $\theta$). However, the higher Ext-vanishing above does not always hold when $\pi_1 \cong \pi_2$, in contrast with Theorem \ref{thm multi standard}. As we saw before in several examples such as the square-integrable representations \cite{CS21, Ch21}, one tends to have more Ext-vanishing phenomenon in the relative setting.

From the representation-theoretic viewpoint, Theorem \ref{thm multi standard} is expected to have an application on determining the layer of the Bernstein-Zelevinsky filtration that contributes to some quotient branching laws \cite{Ch22+}.

Theorem \ref{thm multi standard} does not hold if one replaces standard modules by arbitrary parabolically induced modules. For example, $\nu^{-1/2}\times \nu^{1/2}$ admits the trivial $\mathrm{GL}_1(F)$-character as a quotient with multiplicity $2$. Here $\nu(g)=|\mathrm{det}(g)|_F$. This is due to that the $G_1$-Iwahori component of  $\nu^{-1/2}\times \nu^{1/2}$ is not indecomposable. In such sense, Theorem \ref{thm multi standard} is compatible with \cite[Conjecture  6.5]{Ch21} (more precisely, a variation of the conjecture discussed in {\it loc. cit}).

One may of course expect a similar result should hold for other classical groups, which extends a conjecture of D. Prasad \cite[Conjecture 7.1]{Pr18}. Such extension would also be important for a conjectural relation between an integral formula and an Euler-Poincar\'e pairing in \cite{Pr18}. There are Hom-multiplicity results for other classical groups such as the work of M\oe glin-Waldspurger, Beuzart-Plessis, Loeffler \cite{MW12, BP16, Lo21}. 

We remark that an element in the Hom space in Theorem \ref{thm multi standard} is explicitly constructed in \cite{JPSS83} via Rankin-Selberg integrals. Theorem \ref{thm multi standard} could also be directly generalized to other Bessel and Fourier-Jacobi models by using a Gan-Gross-Prasad type reduction \cite{Ch20}. 

Since any irreducible representation appears in a quotient of a standard representation, we have an alternate proof for the multiplicity theorem for irreducible ones:

\begin{corollary} \cite{AGRS10} \label{cor mult one original}
Let $F$ be non-Archimedean. Let $\pi_1$ and $\pi_2$ be irreducible smooth representations of $G_{n+1}$ and $G_n$ respectively. Then 
\[ \mathrm{dim}~\mathrm{Hom}_{G_n}(\pi_1, \pi_2)  \leq 1 .
\]
\end{corollary}

The main ingredients of the proof for Theorem \ref{thm multi standard} include the theory of Bernstein-Zelevinsky and the Langlands classification, and thus our method is different from \cite{AGRS10}. We remark that our proof uses the multiplicity one theorem for Whittaker models \cite{GK71, Sh74}, and works more directly for all characteristics (e.g. c.f. \cite{Me21}). While the proof in \cite{CS21} also uses a form of left and right filtrations, the argument is not the same as the one presented in this paper and, for example, we did not use Proposition \ref{prop dual restrict} at that time. The proof of this article is a hybrid of \cite{Ch21} and \cite{Ch20} (also c.f. \cite{MW12}).





A main ingredient of the proof is the notion of left-right Bernstein-Zelevinsky filtration. The study of left and right filtrations already has applications in several branching problems such as indecomposability and submodule structure \cite{CS21, Ch21}. 

We also study the standard module in the context of an Euler-Poincar\'e pairing formula due to D. Prasad, which is of independent interest. The standard module gives a new basis in the Grothendieck group of the category of smooth representations of $G_n$. We deduce a consequence on the sum of the coefficients involved in the change of basis in Section \ref{ss kazhdan lusztig coe}.

It follows from Harish-Chandra's Lefschetz principle \cite{HC70}, the picture for real groups and $p$-adic groups should be closely related. Motivated from this, we study the Bruhat filtration \cite{CHM00} and its left-right version in Section \ref{ss filtrion from parabolic}, and use these to prove a Hom-multiplicity result for equal rank Fourier-Jacobi models for principal standard modules in Section \ref{ss archi case multi} when $F=\mathbb R$ and in Section \ref{ss multi complex} when $F=\mathbb C$. The multiplicity one for irreducible representations of the equal rank Fourier-Jacobi model is obtained by Liu-Sun \cite{LS13}. Precisely, we have:

\begin{theorem} \label{thm multi standard real}
Let $F=\mathbb{R}$ or $\mathbb{C}$. Let $\pi_1$ and $\pi_2$ be standard principal series of $\mathrm{GL}_n(F)$. Then 
\[  \mathrm{dim}~\mathrm{Hom}_{\mathrm{GL}_n(F)}(\pi_1 \widehat{\otimes} \mathcal S(F^n), \pi_2^{\vee}) =1.
\]
Here $\mathcal S(F^n)$ is the space of Schwartz functions on $F^n$ with a natural $\mathrm{GL}_n(F)$-action.
\end{theorem}

We remark that we essentially use \cite{LS13} in Theorem \ref{thm multi standard real}, in constrast with the $p$-adic group case.

\subsection{Acknowledgments} The author would like to thank Dragan Mili\v{c}i\'c, Dipendra Prasad, Gordan Savin and Peter Trapa for useful communications. The author would like to thank the referee for useful remarks and suggestions to improve expositions of the article.




\section{Preliminaries}
\subsection{Notations} \label{ss notations}

Let $G_n=\mathrm{GL}_n(F)$, where $F$ is a non-Archimedean local field. Denote by $\mathrm{Irr}(G_n)$ the set of irreducible representations of $G_n$. Denote by $\mathrm{Irr}^c(G_n)$ the set of irreducible cuspidal representations of $G_n$. Let $\mathrm{Irr}=\sqcup_n \mathrm{Irr}(G_n)$ and let $\mathrm{Irr}^c=\sqcup_n\mathrm{Irr}^c(G_n)$. Let $\mathrm{Alg}(G_n)$ be the category of smooth representations of $G_n$. For $\pi \in \mathrm{Alg}(G_n)$, set $n(\pi)=n$, and set $\pi^{\vee}$ to be the smooth dual of $\pi$. Let $\nu(g)=|\mathrm{det} g|_F$ for $g \in G_n$.

For a closed subgroup $H$ of $G$ and $\pi \in \mathrm{Alg}(H)$, we denote by $\mathrm{Ind}_H^G\pi$ the normalized parabolic induction, and denote by $\mathrm{ind}_H^G\pi$ the normalized parabolic induction with compact support. 

A Zelevinsky segment is a datum of the form 
\[ \Delta=[\nu^a\rho, \nu^b\rho]=\left\{ \nu^a\rho, \nu^{a+1}\rho, \ldots, \nu^b \rho \right\},\] where $\rho \in \mathrm{Irr}^c$ and $a, b \in \mathbb{C}$ with $b-a \in \mathbb Z_{\geq 0}$. The absolute (resp. relative) length of $\Delta$ is defined as $(b-a+1)n(\rho)$ (resp. $b-a+1$). We set $a(\Delta)=\nu^a\rho$ and $b(\Delta)=\nu^b\rho$. For a segment $\Delta$, we associate an irreducible representation $\mathrm{St}(\Delta)$, which is the unique irreducible quotient of 
\[  \nu^a \rho \times \ldots \times \nu^b \rho .
\]

For $\rho_1, \rho_2 \in \mathrm{Irr}^c$, we write $\rho_1 >\rho_2$ if there exists $c \in \mathbb{Z}_{>0}$ such that $\rho_1 \cong \nu^c \rho_2$. For $\Delta_1, \Delta_2$, we write $\Delta_1>\Delta_2$ if $b(\Delta_1)>b(\Delta_2)$. Two segments $\Delta_1, \Delta_2$ are linked if $\Delta_1 \cup \Delta_2$ is still a segment, and $\Delta_1 \not\subset \Delta_2$, and $\Delta_2 \not\subset \Delta_1$. Otherwise, we say the segments are unlinked.

Let $\pi_1 \in \mathrm{Alg}(G_{n_1})$ and let $\pi_2 \in \mathrm{Alg}(G_{n_2})$. We fix a parabolic subgroup $P=LN$ (resp. $\overline{P}=L\overline{N}$), where $L$ is the Levi subgroup containing matrices $\mathrm{diag}(g_1, g_2)$ for $g_i \in G_{n_i}$ and $N$ (resp. $\overline{N}$) is the unipotent radical contained in the subgroup of upper (resp. lower) triangular matrices. 

We define the products:
\[  \pi_1 \times \pi_2 = \mathrm{Ind}_P^G \pi_1 \boxtimes \pi_2,
\]
where $\pi_1 \boxtimes \pi_2$ extends trivially to $P$. 

A multisegment $\mathfrak m$ is a multiset of segments. Denote by $\mathrm{Mult}_n$ the set of multisegments, for which the sum of absolute lengths of segments is equal to $n$. Let $\mathrm{Mult}=\sqcup_n \mathrm{Mult}_n$. For a multisegment $\mathfrak m = \left\{ \Delta_1, \ldots, \Delta_r \right\}$, we relabel such that for $i <j$, $\Delta_i \not <\Delta_j$. Then the Langlands classification shows that there exists a unique simple quotient, denoted by $\mathrm{St}(\mathfrak m)$, of the induced representation:
\[  \lambda(\mathfrak m):= \mathrm{St}(\Delta_1) \times \ldots \times \mathrm{St}(\Delta_r)
\]
We shall also refer those $\lambda(\mathfrak m)$ to be {\it standard modules}.

For $\Delta=[\nu^a\rho, \nu^b\rho]$, define 
\[  \Delta^{\vee} = [\nu^{-b}\rho^{\vee}, \nu^{-a}\rho^{\vee}] 
\]
We have that
\[ \mathrm{St}(\Delta)^{\vee} \cong \mathrm{St}(\Delta^{\vee}) .
\]

For each irreducible $\pi$, there exists a unique multiset $\left\{ \rho_1, \ldots, \rho_k \right\}$ with each $\rho_i \in \mathrm{Irr}^c$ such that $\pi$ is a composition factor of $\rho_1\times \ldots \times \rho_k$. We shall denote such set by $\mathrm{csupp}(\pi)$, and call it the cuspidal support of $\pi$. For a representation $\pi$ of finite length, if all the simple composition factors in $\pi$ have the same cuspidal support, then we also write $\mathrm{csupp}(\pi)$ to represent the cuspidal support of any simple composition factor in $\pi$.

Let $\theta: G_n \rightarrow G_n$ given by $\theta(g)=g^{-t}$ be the Gelfand-Kazhdan involution \cite{BZ76}. It induces a self-equivalence exact functor on $\mathrm{Alg}(G_n)$, still denoted by $\theta$.

Let $U_i \subset G_i$ be the subgroup of all upper triangular matrices and let $\psi_i$ be a generic character on $U_i$. Define 
\[R_i=\left\{ \begin{pmatrix} I_{n-i} & m \\ & u \end{pmatrix} : m \in \mathrm{Mat}_{{n-i}\times i}, u \in U_i \right\} \]
\[ \psi': R_i \rightarrow \mathbb{C}; \quad \psi'(\begin{pmatrix} I_{n-i} & m \\ & u \end{pmatrix} )=\psi_i(u) . \]
The right $i$-th Bernstein-Zelevinsky derivative is defined as:
\[ \pi^{(i)} = \delta^{-1/2} \frac{\pi}{\langle n.x-\psi'(n)x : x \in \pi, n \in R_i \rangle },
\]
where $\delta$ is the modular character for $R_i$. We also define
\[  {}^{(i)}\pi=\theta(\theta(\pi)^{(i)}), \quad \pi^{[i]}=\nu^{1/2}\cdot\pi^{(i)}, \quad {}^{[i]}\pi=\nu^{-1/2}\cdot{}^{(i)}\pi,
\]
which will be called left, shifted right and shifted left derivatives respectively. Note that when $i=0$, $\pi^{[0]}$ is just $\nu^{1/2}\otimes \pi$.

One of basic and important computations for derivatives is \cite{Ze80}
\begin{align} \label{eqn steinberg derivative}
  \mathrm{St}(\Delta)^{(i)} =\mathrm{St}({}^{(i)}\Delta), \quad {}^{(i)}\mathrm{St}(\Delta)=\mathrm{St}(\Delta^{(i)}),
\end{align}
where ${}^{(i)}\Delta$ is obtained by truncating $\Delta$ from left $i$-times and $\Delta^{(i)}$ is obtained by truncating $\Delta$ from right $i$-times.

Let $M_n\subset G_n$ be the mirabolic subgroup of $G_n$ i.e. containing all matrices with the last row $(0,\ldots, 0,1)$. We have a sequence of embeddings:
\begin{align} \label{eqn subgroup relations}
   G_0 =M_1 \subset G_1\subset  \ldots \subset G_{n-1} \subset M_n \subset G_n,
\end{align}
where $G_{i-1}$ is viewed as a subgroup of $M_i$ via the embedding $g \mapsto \begin{pmatrix} g & 0 \\ & 1 \end{pmatrix}$. We shall view $G_i, M_i$ $(i \leq n)$ as a subgroup of $G_n$ via the above embeddings.

\subsection{Dual restriction}

\begin{proposition} \label{prop dual restrict} (c.f. \cite{MW12})
Let $\pi_1 \in \mathrm{Alg}(G_{n+1})$ and $\pi_2 \in \mathrm{Alg}(G_n)$. Then 
\[  \mathrm{Ext}^i_{G_{n}}(\pi_1, \pi_2^{\vee}) \cong \mathrm{Ext}^i_{G_{n+1}}(\sigma \times \pi_2, \pi_1^{\vee})
\] 
for some cuspidal representation $\sigma$ of $\mathrm{GL}_2(F)$. 
\end{proposition}

\begin{proof}
This is shown in \cite[Proposition 4.1]{Ch20} and we only sketch it. The main idea is that $(\sigma \times \pi_2)|_{G_{n+1}}$ contains $\mathrm{ind}_{G_n}^{G_{n+1}}\pi_2$ as a submodule for a suitable choice of $\sigma$. One can show that $\mathrm{ind}_{G_n}^{G_{n+1}}\pi_2$ is the only composition factor contributing the non-zeroness of  $\mathrm{Ext}^i_{G_{n+1}}(\sigma \times \pi_2, \pi_1^{\vee})$. One then applies duals and Frobenius reciprocity (see \cite{Pr18}) to get to Ext in the LHS form.
\end{proof}

\subsection{Parabolic inductions}

A simple but useful formula is the following:
\begin{lemma} \label{lem GK involution on product}
For $\pi_1 \in \mathrm{Alg}(G_{n_1})$ and $\pi_2 \in \mathrm{Alg}(G_{n_2})$,
\[ \theta(\pi_1 \times \pi_2) \cong \theta(\pi_2) \times  \theta(\pi_1) 
\]
\end{lemma}


\subsection{Rankin-Selberg models}
Let $U_r$ be the subgroup of $G_r$ containing all unipotent upper triangular matrices of $G_r$. For $r \geq 0$, let
\[ H^R_{r,n} :=\left\{ \begin{pmatrix} g & 0 & x \\ & 1 & v^t \\ & & u \end{pmatrix} : g \in G_{n-r}, x \in M_{n-r,r}, u \in U_r, v \in F^r \right\} \subset G_{n+1} ,
\]
and let $\bar{H}^R_{r,n}$ is the transpose of $H^R_{r,n}$. Define $\zeta_R: H^R_{r,n}\rightarrow \mathbb{C}$ as
\[ \zeta_R(\begin{pmatrix} g & 0 & x \\ & 1 & v^t \\ & & u \end{pmatrix})=\psi(\begin{pmatrix} 1 & v^t  \\  & u \end{pmatrix} ),
\]
where $\psi$ is a non-degenerate character on $U_{r+1} \cong F^r \rtimes U_r$. (We remark that the inductions here are the normalized ones and so we do not have a normalizing factor  c.f. \cite[Section 5.2]{Ch20}.)

For $\pi \in \mathrm{Alg}(G_{n-r})$, extend $\pi$ to an $H^R_{r,n}$-representation by letting the unipotent radical of $H^R_{r,n}$ acting trivially. The Rankin-Selberg model is defined as:
\[ \mathrm{RS}_r(\pi)= \mathrm{ind}_{H^R_{r,n}}^{G_{n+1}} \pi \otimes \zeta^R
\]
and 
\[ \overline{\mathrm{RS}}_r(\pi) =\mathrm{ind}_{\bar{H}^R_{r,n}}^{G_{n+1}}\pi \otimes \theta(\zeta^R).
\]

For $\pi \in \mathrm{Alg}(G_{n+1})$ and $\pi' \in \mathrm{Alg}(G_{n-r})$, 
\[  m^i(\pi, \pi'{}^{\vee})=\mathrm{dim}~ \mathrm{Ext}^i_{H_{r,n}^R}(\pi \otimes \zeta^R, \pi'{}^{\vee})
\]

We also have the following Gan-Gross-Prasad type reduction for Ext-groups (see \cite{Ch20}).

\begin{proposition} \label{prop multiplicity dual eqn}
Let $\pi \in \mathrm{Alg}(G_{n+1})$ and $\pi' \in \mathrm{Alg}(G_{n-r})$. Then
\[  m^i(\pi, \pi'{}^{\vee}) = m^i(\pi' \times \sigma, \pi^{\vee})
\]
for some $\sigma \in \mathrm{Irr}^c(G_{r+1})$. 
\end{proposition}

\subsection{Fourier-Jacobi models}

Define $\zeta^F$(resp. $\widehat{\zeta}^F$) to be the space $S(F^n)$ of Schwartz functions on $F^n$ with the action of $G_n$ given by: for $f \in S(F^n)$, 
\[  (g.f)(v) =\nu^{-1/2}(g)f(g^{-1}.v) 
\]
\[ (\mbox{resp.  $\quad (g.f)(v) = \nu^{1/2}(g)f(g^t.v)$} . )
\]
In particular, $\theta(\zeta^F) \cong \widehat{\zeta}^F$. Indeed, as shown in \cite[Proposition 6.1]{Ch20}, $\widehat{\zeta}^F\cong \zeta^F$.

For $\pi \in \mathrm{Alg}(G_n)$, we shall refer $\pi \otimes \zeta^F$ to the Fourier-Jacobi model of $\pi$ with equal rank \cite{GGP12}. 


\subsection{Left-right filtrations: parabolic form}



We now give the left and right filtrations that we need (c.f. \cite[Sections 5 and 6]{Ch21}):

\begin{proposition} \label{prop right filtration p} \cite[Section 5]{Ch20}
Let $\pi_1 \in \mathrm{Alg}(G_{n_1})$ and let $\pi_2 \in \mathrm{Alg}(G_{n_2})$. Then $(\pi_1 \times \pi_2)|_{G_{n_1+n_2-1}}$ admits a filtration with successive quotients:
\[  \pi_1^{[0]} \times (\pi_2|_{G_{n_2-1}}), \quad  \pi_1^{[1]} \times (\pi_2\otimes \zeta^F)
\]
and for $k \geq 2$, 
\[  \pi_1^{[k]} \times \mathrm{RS}_{k-2}(\pi_2)
\]
\end{proposition}

\begin{proof}
We only sketch the proof. One first applies the exact sequence in \cite[Proposition 4.13(a)]{BZ77} to obtain the term $\pi_1^{[0]} \times (\pi_2|_{G_{n_2-1}})$ for the top layer, and then applies the Bernstein-Zelevinsky filtration \cite[Proposition 3.2(e)]{BZ77}. One also has to use \cite[Lemma 3.4]{Ch20} to change to the stated form.
\end{proof}

The 'left' version is as follows:

\begin{proposition} \label{prop left filtration}
Let $\pi_1 \in \mathrm{Alg}(G_{n_1})$ and let $\pi_2 \in \mathrm{Alg}(G_{n_2})$. Then $(\pi_1 \times \pi_2)|_{G_{n_1+n_2-1}}$ admits a filtration with successive subquotients:
\[  (\pi_1|_{G_{n_1-1}}) \times {}^{[0]}\pi_2, \quad  (\pi_2 \otimes \widehat{\zeta}^F) \times {}^{[1]}\pi_2
\]
and for $k \geq 2$,
\[  \overline{\mathrm{RS}}_{k-2}(\pi_2) \times {}^{[k]} \pi_2 
\]
\end{proposition}

\begin{proof}
Set $G=G_{n_1+n_2-1}$. We first consider the restriction $\theta((\pi_1 \times \pi_2)|_{G})$. Since 
\[  \theta((\pi_1 \times \pi_2 )|_{G}) \cong \theta(\pi_1 \times\pi_2)|_{G},
\]
Lemma \ref{lem GK involution on product} gives
\[ \theta((\pi_1 \times \pi_2 )|_{G})  \cong (\theta(\pi_2)\times \theta(\pi_1))|_G .
\]
By Proposition \ref{prop right filtration p}, we obtain a filtration with successive subquotients of the form:
\[ \theta(\pi_2)^{[0]} \times (\theta(\pi_1)|_{G_{n_1-1}}), \quad \theta(\pi_2)^{[1]} \times (\theta(\pi_1)\otimes \zeta^F)
\]
\[ \theta(\pi_2)^{[k]} \times (\mathrm{RS}_{k-2}(\theta(\pi_1))).
\]
Now applying $\theta$-action again, we obtain the desired filtration.
\end{proof}

\section{Proof of Theorem \ref{thm multi standard}}

\subsection{Proof of Theorem \ref{thm multi standard}}

Let $\pi=\lambda(\mathfrak m)$ be a standard representation for some $\mathfrak m=\left\{ \Delta_1, \ldots, \Delta_r\right\} \in \mathrm{Mult}_n$. Similarly, let $\pi'=\lambda(\mathfrak m')$ be a standard representation for $\mathfrak m'=\left\{ \Delta'_1, \ldots, \Delta'_s \right\} \in \mathrm{Mult}_{n-1}$. 

Our goal is to show that
\[  \mathrm{Hom}_{G_n}(\pi, \pi'{}^{\vee}) \cong \mathbb{C}
\]
and, for $i \geq 1$,
\[  \mathrm{Ext}^i_{G_n}(\pi, \pi'{}^{\vee}) =0
\]

Two segments $\Delta$ and $\Delta'$ are said to be in the same cuspidal line if $(\nu^a\Delta) \cap \Delta' \neq \emptyset$ for some $a \in \mathbb Z$. Let $L^*:=L(\nu^{1/2}\mathfrak m, \mathfrak m'{}^{\vee})$ be the number of pairs $(\Delta, \Delta')$ in $(\nu^{1/2}\mathfrak m) \times \mathfrak m'{}^{\vee}$ such that $\Delta$ and $\Delta'$ are in the same cuspidal line. We shall prove the theorem by an induction on $L^*$. 

Let $\Pi_n$ be the Gelfand-Graev representation of $G_n$ (see \cite{CS19}). When $L^*=0$, this follows from the Bernstein-Zelevinsky filtration and comparing central characters that:
\[  \mathrm{Ext}^i_{G_n}(\pi , \pi'{}^{\vee}) \cong \mathrm{Ext}^i_{G_n}(\Pi_n, \pi'{}^{\vee})
\] 
and we omit the details in this case (see \cite{CS21,Ch21,Ch20} for similar computations). Now the vanishing higher Exts for the latter term follow from that $\Pi_n$ is projective \cite{CS19}, and the multiplicity one for Hom follows from the multiplicity one for Whittaker models \cite{GK71} (also see \cite{Sh74}).

We now assume that $L^*>0$. Let $\rho \in \mathrm{Irr}^c$ such that $\rho \in \nu^{1/2}\Delta_l$ for some $l$ and $\nu^{c}\rho \in \Delta_k'$ for some $k$. We choose a maximal $\rho^* \in \mathrm{csupp}_{\mathbb Z}(\rho)$ in $\mathrm{csupp}(\nu^{1/2}\pi) \cup \mathrm{csupp}(\pi'{}^{\vee})$, where the ordering is $<$ defined in Section \ref{ss notations}. It is a simple observation that the maximality implies $\rho^* \cong \nu^{1/2}b(\Delta_l)$ for some $l$ or $\rho^* \cong b(\Delta_k'{}^{\vee})$ for some $k$. \\

\noindent
{\bf Case 1:} $ \rho^* \not\cong b(\Delta_p'{}^{\vee})$ for any $p$ and $\rho^* \cong \nu^{1/2}b(\Delta_l)$ for some $l$. Relabelling if necessary, we may and shall assume that $l=1$.
The maximality of $\rho^*$ gives that
\begin{equation} \label{eqn non support condition case 1}
\nu^{1/2}b(\Delta_1) \notin \mathrm{csupp}(\lambda(\mathfrak m')^{\vee}). 
\end{equation}

It follows from \cite{Ze80} that $\mathrm{St}(\Delta) \times \mathrm{St}(\Delta') \cong \mathrm{St}(\Delta')\times \mathrm{St}(\Delta)$ for any two unlinked segments $\Delta$ and $\Delta'$. Hence, we can write
\[  \lambda(\mathfrak m)=\mathrm{St}(\Delta_1) \times \lambda(\widetilde{\mathfrak m}),
\]
where $\widetilde{\mathfrak m} =\mathfrak m\setminus \left\{ \Delta_1 \right\}$. 


In this case, we observe that the cuspidal support of $\mathrm{St}(\Delta_1)^{[0]}$ contains $\nu^{1/2}b(\Delta_1)$. Thus,
\[  \mathrm{Ext}^i_{G_n}(\mathrm{St}(\Delta_1)^{[0]} \times (\lambda(\widetilde{\mathfrak m})|_{G_{m-1}}), \lambda(\mathfrak m')^{\vee}) \cong\mathrm{Ext}^i(\mathrm{St}(\Delta_1)^{[0]}\boxtimes (\lambda(\widetilde{\mathfrak m})|_{G_{m-1}}), (\lambda(\mathfrak m')^{\vee})_{N^-}) =0 ,
\]
where $m=n(\lambda(\widetilde{\mathfrak m}))$, the first isomorphism follows from second adjointness of Frobenius reciprocity \cite{Pr18} with a certain opposite unipotent radical $N^-$, and the second equality follows from comparing cuspidal support with (\ref{eqn steinberg derivative}) (at the point $\nu^{1/2}b(\Delta_1)$). Similar argument gives that for all $i$,
\[  \mathrm{Ext}^i_{G_n}(\mathrm{St}(\Delta_1)^{[1]}\times (\lambda(\widetilde{\mathfrak m})\otimes \zeta^F), \lambda(\mathfrak m')^{\vee})=0 ,
\]
and $2 \leq a <i^*$,
\[  \mathrm{Ext}^i_{G_n}(\mathrm{St}(\Delta_1)^{[a]}\times \mathrm{RS}_{a-2}(\lambda(\widetilde{\mathfrak m})), \lambda(\mathfrak m')^{\vee})=0 .
\]
Thus a standard long exact sequence argument reduces to:
\[ \mathrm{Ext}^i_{G_n}(\mathrm{St}(\Delta_1)\times \lambda(\mathfrak m), \lambda(\mathfrak m')^{\vee}) \cong \mathrm{Ext}^i_{G_n}(\mathrm{RS}_{i^*-2}(\lambda(\mathfrak m)), \lambda(\mathfrak m')^{\vee}) 
\]
The latter term is equal to $m^i(\sigma \times \lambda(\mathfrak m), \lambda(\mathfrak m')^{\vee})$ for some $\sigma \in \mathrm{Irr}^c$ with $\sigma \notin \mathrm{csupp}_{\mathbb{Z}}(\lambda(\mathfrak m))\cup \mathrm{csupp}_{\mathbb Z}(\nu^{-1/2}\lambda(\mathfrak m')^{\vee})$ by Proposition \ref{prop multiplicity dual eqn} and Frobenius reciprocity. Now by induction on $N^*$, we have that $m^i(\sigma \times \lambda(\mathfrak m), \lambda(\mathfrak m')^{\vee})=0$ for $i >0$ and $=1$ for $i=0$.  \\


{\bf Case 2:} $\rho^* \cong b(\Delta_k'{}^{\vee})$ for some $k$. It follows from the maximality of $\rho^*$ that $\nu^{1/2}b(\Delta_k'{}^{\vee}) \notin \mathrm{csupp}(\lambda(\mathfrak m))$. Rephrasing the condition, it is equivalent to:
\begin{equation} \label{eqn case 2 main}  \nu^{-1/2}a(\Delta_k') \notin \mathrm{csupp}(\lambda(\mathfrak m)^{\vee})
\end{equation}

Now, by Proposition \ref{prop dual restrict}, it suffices to show that 
\[  \mathrm{Ext}^i_{G_{n+1}}(\sigma \times \lambda(\mathfrak m'), \lambda(\mathfrak m)^{\vee})\cong \left\{ \begin{array}{c c} \mathbb C & \mbox{ if $i=0$ } \\ 0 & \mbox{ if $i>0$ } \end{array} \right.
\]
for some cuspidal $\sigma \not\in \mathrm{csupp}_{\mathbb Z}(\nu^{-1/2}\lambda(\mathfrak m)^{\vee})\cup \mathrm{csupp}_{\mathbb Z}(\lambda(\mathfrak m'))$. Note that $\sigma \times \lambda(\mathfrak m')$ is still a standard module $\lambda(\mathfrak m'')$ with $\mathfrak m''=\mathfrak m' \cup \left\{ [\sigma] \right\}$. Note that 
\begin{align*}
 |\mathrm{csupp}(\nu^{1/2}\lambda(\mathfrak m''))\cap \mathrm{csupp}(\lambda(\mathfrak m)^{\vee})| &=|\mathrm{csupp}(\nu^{1/2}\lambda(\mathfrak m'))\cap \mathrm{csupp}(\lambda(\mathfrak m)^{\vee})| \\
  & = |\mathrm{csupp}(\nu^{-1/2}\lambda(\mathfrak m')^{\vee})\cap \mathrm{csupp}(\lambda(\mathfrak m))| \\
	& = |\mathrm{csupp}(\lambda(\mathfrak m')\cap \mathrm{csupp}(\nu^{1/2}\lambda(\mathfrak m))| \\
	&= L^*
	\end{align*}
Similar as in Case 1, we have that 
\[  \lambda(\mathfrak m') \cong \lambda(\widetilde{\mathfrak m}')\times \mathrm{St}(\Delta_k'),
\]
where $\widetilde{\mathfrak m}'=\mathfrak m' \setminus \left\{ \Delta_k' \right\}$. Now we can show (\ref{eqn case 2 main}) by a similar inductive argument as in Case 1 with the replacement of (\ref{eqn non support condition case 1}) by (\ref{eqn case 2 main}), and the replacement of Proposition \ref{prop right filtration p} by Proposition \ref{prop left filtration}.




\subsection{Proof of Corollary \ref{cor mult one original}}
Let $\pi_1 \in \mathrm{Irr}(G_{n+1})$ and let $\pi_2\in \mathrm{Irr}(G_n)$. Let $\mathfrak m_1, \mathfrak m_2 \in \mathrm{Mult}$ such that $\pi_1 \cong \mathrm{St}(\mathfrak m_1)$ and $\pi_2 \cong \mathrm{St}(\mathfrak m_2)^{\vee}$ respectively. Then we have a surjection 
\[  \lambda(\mathfrak m_1) \rightarrow \pi_1 ,
\]
and an injection 
\[  \pi_2 \hookrightarrow \lambda(\mathfrak m_2)^{\vee} .
\]
This gives that 
\[  \mathrm{dim}~\mathrm{Hom}_{G_n}(\pi_1, \pi_2) \leq \mathrm{dim}~\mathrm{Hom}_{G_n}(\lambda(\mathfrak m_1), \lambda(\mathfrak m_2)^{\vee}) = 1 .
\]

\subsection{Generic parabolically induced modules}

As mentioned in introduction, Theorem \ref{thm multi standard} does not hold for general generic parabolically induced modules. Nevertheless, we have the following result c.f. \cite{JPSS83, JS83}:

\begin{corollary} \label{cor non-zero generic para}
Let $\Delta_1, \ldots , \Delta_k$ be segments with sum of their absolute lengths equal to $n+1$, and let $\Delta_1', \ldots, \Delta_l'$ be segments with sum of their absolute lengths equal to $n$. Let
\[ \lambda =\mathrm{St}(\Delta_1)\times \ldots \times \mathrm{St}(\Delta_k) ,
\]
\[  \lambda' =\mathrm{St}(\Delta_1')\times \ldots \times \mathrm{St}(\Delta_l') .
\]
Then 
\[ \mathrm{Hom}_{G_n}(\lambda, \lambda'{}^{\vee}) \neq 0 .
\]
\end{corollary}

\begin{proof}

\noindent
{\it Claim}: $\mathrm{St}(\Delta_1)\times \ldots \times \mathrm{St}(\Delta_k)$ has a quotient isomorphic to a standard representation. Similarly, $\mathrm{St}(\Delta_1')\times \ldots \times \mathrm{St}(\Delta_l')$ has a quotient isomorphic to a standard representation. \\

Note that the corollary will follow from the claim and Theorem \ref{thm multi standard}. It remains to prove the claim. We only prove the first assertion of the claim and the latter one can be proved similarly.

Two segments $\Delta^1$ and $\Delta^2$ are said to be in opposite order if $\Delta^1$ and $\Delta^2$ are linked; and $b(\Delta^1) <b(\Delta^2)$. Let $N$ be the number of pairs $(\Delta_i, \Delta_j)$ such that $\Delta_i$ and $\Delta_j$ are in opposite order. If $N=0$, then the modules themselves are already standard.   \\

Suppose $N>0$. Now we choose the smallest $i^*$ such that $\Delta_{i^*}$ and $\Delta_{i^*+1}$ are in opposite order. Then we replace the segments $\Delta_{i^*}$ and $\Delta_{i^*+1}$ by the segments $\Delta=\Delta_{i^*}\cup \Delta_{i^*+1}$ and $\overline{\Delta}=\Delta_{i^*}\cap \Delta_{i^*+1}$. Using \cite{Ze80}, we have a surjection:
\[ \lambda \rightarrow \mathrm{St}(\Delta_1)\times \ldots \times \mathrm{St}(\Delta_{i^*-1})\times \mathrm{St}(\Delta) \times \mathrm{St}(\overline{\Delta}) \times \mathrm{St}(\Delta_{i^*+2})\times \ldots \times \mathrm{St}(\Delta_r) .
\]
(As convention, we set $\mathrm{St}(\emptyset)=1$ as a $G_0$-representation.) If any pair of the segments in the last term is not in opposite order, then the restriction of the last term has a non-zero map to $\lambda'{}^{\vee}$ by Theorem \ref{thm multi standard} and so does $\lambda|_{G_n}$. 

If some pair of the segments in the last is in opposite order, then we repeat the union-intersection process on the segments above. This process will terminate (e.g. considering the ordering in \cite[Chapter 7]{Ze80}). It means that we showed that $\lambda$ has a quotient such that any two segements are not in opposite order. Hence we are done for the claim.

\end{proof}

\begin{remark}
As shown in the proof of Corollary \ref{cor non-zero generic para}, the non-zero Hom comes from Hom between the images of $\lambda$ and $\lambda'$ embedded to the Whittaker model. Hence, the non-zero element Hom in the above proof indeed agrees with the one in \cite{JPSS83}.
\end{remark}

\subsection{Zelevinsky standard representations}

For a segment $\Delta=[\nu^a\rho, \nu^b\rho]$, let $\langle \Delta \rangle$ be the unique quotient of 
\[  \nu^{a}\rho \times \ldots \times \nu^b\rho .
\]
Let $\mathfrak m=\left\{ \Delta_1, \ldots, \Delta_r \right\} \in \mathrm{Mult}$ with the labelling satisfying that for $i<j$, $\Delta_i \not< \Delta_j$. The representation
\[  \langle \Delta_1 \rangle \times \ldots \times \langle \Delta_r \rangle
\]
is called a Zelevinsky standard representation. One can prove the following by a similar method as in \cite[Proposition 2.3]{Ch21} and we omit the details:

\begin{proposition} \label{prop embedding std module}
Let $\mathfrak m \in \mathrm{Mult}$. Then there exists a multisegment $\mathfrak m'$ whose all segments are singletons such that $\zeta(\mathfrak m')$ maps onto $\zeta(\mathfrak m)$. 
\end{proposition}

\begin{corollary}
Let $\pi$ and $\pi'$ be Zelevinsky standard representations of $G_{n+1}$ and $G_n$ respectively. Then
\[ \mathrm{dim}~\mathrm{Hom}_{G_n}(\pi, \pi'{}^{\vee}) \leq 1 
\]
\end{corollary}

\begin{proof}
This follows from Theorem \ref{thm multi standard}, Proposition \ref{prop embedding std module} and the fact \cite{Ze80} that $\zeta(\mathfrak m')\cong \lambda(\mathfrak m')$ if $\mathfrak m'$ is a multisegment whose all segments are singletons.
\end{proof}

\begin{remark}
It is easy to construct examples (e.g. using characters) that 
\[ \mathrm{Hom}_{G_n}(\pi, \pi'{}^{\vee})=0,\] 
and examples that $\mathrm{Hom}_{G_n}(\pi, \pi'{}^{\vee})\cong \mathbb C$, where $\pi$ and $\pi'$ are Zelevinsky standard representations of $G_{n+1}$ and $G_n$.

The higher Ext for Zelevinsky standard representations also does not vanish in general. For example, let $\pi$ and $\pi'$ be trivial representations of $G_{n+1}$ and $G_n$ respectively. Then $\mathrm{Ext}^1_{G_n}(\pi, \pi') \cong \mathbb{C}$. 
\end{remark}



\section{Euler-Poincar\'e pairing and standard modules} \label{ss kazhdan lusztig coe}

\subsection{Euler-Poincar\'e pairing}

Let $\mathrm{Alg}_f(G_n)$ be the full subcategory of $\mathrm{Alg}(G_n)$ of representations of finite length. Let $\pi \in \mathrm{Alg}_f(G_n)$. We shall write $[\pi]$ for the corresponding element in the Grothendieck group $K(\mathrm{Alg}_f(G_n))$. It follows from the Langlands classification that standard modules form a basis for $K(\mathrm{Alg}_f(G_n))$. For $\pi \in \mathrm{Irr}(G_n)$, write
\begin{align} \label{eqn gorth express} [\pi] = \sum_{\pi' \in \mathrm{Irr}(G_n)} m_{\pi, \pi'} [I(\pi')] \quad (m_{\pi, \pi'}\in \mathbb Z) ,
\end{align}
where $I(\pi')$ be the corresponding standard module with the irreducible quotient isomorphic to $\pi'$.
We shall call $m_{\pi, \pi'}$ to be a Kazhdan-Lusztig coefficient, which is possibly negative. For a fixed $\pi$, $\pi$ can only appear in a finite number of $I(\pi')$ and so there is only finitely non-zero $m_{\pi,\pi'}$.

We now recall the Euler-Poincar\'e pairing due to Dipendra Prasad:

\begin{theorem}\cite{Pr18} \label{thm euler poincare}
Let $\pi_1 \in \mathrm{Alg}_f(G_{n+1})$ and let $\pi_2 \in \mathrm{Alg}_f(G_n)$. Define
\[  \mathrm{EP}(\pi_1, \pi_2)=\sum_i (-1)^i \mathrm{dim}~ \mathrm{Ext}^i_{G_n}(\pi_1, \pi_2) .
\]
Then $\mathrm{EP}(\pi_1, \pi_2) =\mathrm{dim}~\mathrm{Wh}(\pi_1) \cdot \mathrm{dim}~\mathrm{Wh}(\pi_2)$, where $\mathrm{Wh}(\pi_1)$ (resp. $\mathrm{Wh}(\pi_2)$) are the Whittaker function spaces of $\pi_1$ and $\pi_2$ for $G_{n+1}$ and $G_n$ respectively.
\end{theorem}

The following result is probably well-known, but we think it may be interesting to include it in our content.

\begin{corollary} \label{cor multiplicity kl}
Let $\pi \in \mathrm{Irr}(G_n)$. Then
\begin{enumerate}
\item If $\pi$ is generic, then, for all $\pi' \not\cong \pi$, 
\[ m_{\pi, \pi'}=0
\]
and  $m_{\pi,\pi}=1$.
\item If $\pi$ is not generic, then
\[ \sum_{\pi' \in \mathrm{Irr}(G_n) } m_{\pi, \pi'} =0
\]
\end{enumerate}
\end{corollary}

\begin{proof}
(1) follows from that $\pi=I(\pi)$. We now consider (2) and $n \geq 1$. Let $\pi \in \mathrm{Irr}(G_{n})$ be non-generic and let $\pi_0 \in \mathrm{Irr}(G_{n-1})$ be generic. Using Prasad's Euler-Poincar\'e pairing, we have that:
\[ \mathrm{EP}(\pi,\pi_0):= \sum_{i} (-1)^i\mathrm{dim}~ \mathrm{Ext}^i_{G_n}(\pi, \pi_0) =0
\]
Since the Euler-Poincar\'e pairing only depends on the image in Grothendieck group, 
\[ \sum_{\pi'\in \mathrm{Irr}(G_n)}m_{\pi,\pi'}\mathrm{EP}(I(\pi'),\pi_0)= \mathrm{EP}(\pi,\pi_0) = 0
\]
But the term $\mathrm{EP}(I(\pi'), \pi_0)=1$ by Theorem \ref{thm euler poincare} and hence the lemma follows.
\end{proof}

We remark that the Kazhdan-Lusztig coefficients for some special cases such as Speh representations \cite{Ta95, CR08, Ba14} and ladder representations \cite{LM14, BC15} are known, but in general it is difficult to compute. 

\begin{remark}
Corollary \ref{cor multiplicity kl} can be proved by taking the (exact) Whittaker-Jacquet functor i.e. $n$-th derivative on the expression in (\ref{eqn gorth express}). The author would like to thank Rapha\"el Beuzart-Plessis and Erez Lapid for communicating on this. 
\end{remark}

\section{Filtration on discrete series} \label{ss archi case}

Now we consider the Archimedean case. Let $G_{n}=\mathrm{GL}_{n}(\mathbb R)$ in this and next two sections. All the representations in remaining sections are Casselman-Wallach representations i.e. smooth Fr\'echet representations with moderate growth (see \cite{BK14}). Again let $\theta: G_n \rightarrow G_n$ by $\theta(g)=g^{-t}$, which has analogous property as the $p$-adic one (see \cite{Ad14}). We consider $G_i$ as a subgroup of $G_n$ via an obvious analogue embedding in (\ref{eqn subgroup relations}).

We shall denote by $\mathcal S(X)$ for the space of complex-valued Schwartz functions on an affine Nash manifold $X$ in the sense of \cite{dCl91}. In particular, when $X$ is compact, $\mathcal S(X)$ is isomorphic to the space of smooth functions on $X$ \cite[1.3.6(i) Exemples]{dCl91}; and when $X \cong \mathbb R^n$, $\mathcal S(X)$ agrees with the usual definition \cite[1.3.6(iii) Exemples]{dCl91}. 

\subsection{A filtration on discrete series of $\mathrm{GL}_2(\mathbb R)$} \label{ss filtration}

Let $B=B_2$ be the Borel subgroup of $G_2$ containing all upper triangular matrices. Define the (normalized) parabolically induced representation:
\[ I(\epsilon,s_1,s_2):= \mathrm{Ind}_{B}^{G_2} \epsilon \left(\chi_{s_1}\boxtimes \chi_{s_2} \right) ,
\]
where $s_1, s_2 \in \mathbb{C}$ such that $\chi_{s_1}(t_1)=|t_1|^{s_1}$, $\chi_{s_2}(t_2)=|t_2|^{s_2}$ and $\epsilon$ is a unitary character of $B$. 

The discrete series can be constructed from the unique submodule of the (normalized) parabolic induced representation $I(\tau_m, \frac{m-1}{2}, \frac{m+1}{2})$ with $\tau_m$ satisfying
\[ \tau_m(z^*)=(-1)^m,
\]
where $z^*=\mathrm{diag}(-1,-1)$ and $m \in \mathbb{Z}_{\geq 2}$. Denote by $\delta(m)$ the corresponding discrete series $(m \geq 2$), which is independent of $\tau_m$. The above construction for $m=1$ gives the limit of discrete series denoted $\delta(1)$. The corresponding Harish-Chandra module of those $\delta(m)$ has $\mathrm{SO}(2)$-types:
\[  \ldots, e^{-\sqrt{-1}(m+2)\theta}, e^{-\sqrt{-1}m\theta},  e^{\sqrt{-1}m\theta}, e^{\sqrt{-1}(m+2)\theta} , \ldots,
\]
where $e^{\sqrt{-1}j\theta}$ represents the $\mathrm{SO}(2)$-representation with the action given by 
\[  \begin{pmatrix} \cos \alpha & \sin \alpha \\ -\sin \alpha & \cos \alpha \end{pmatrix}. v= e^{\sqrt{-1}j\alpha}.v
\]
Since $\theta(I(\epsilon, \frac{m-1}{2}, \frac{m+1}{2}))\cong I(\epsilon, \frac{m-1}{2}, \frac{m+1}{2})$, we have that $\theta(\delta(m))\cong \delta(m)$.

Let $\chi_1=\epsilon_1\otimes \chi_{s_1}$ and $\chi_2=\epsilon_2\otimes\chi_{s_2}$, where $\epsilon_1(t)=\epsilon(\mathrm{diag}(t,1))$ and $\epsilon_2(t)=\epsilon(\mathrm{diag}(1,t))$. By Borel's lemma, we obtain a filtration:
\[ 0 \rightarrow X\rightarrow    I(\epsilon, s_1,s_2) \rightarrow Y  \rightarrow 0
\]
with $Y$ isomorphic to $\mathbb C[[z]]$ with action characterized by 
\[  \begin{pmatrix} t & \\ & 1 \end{pmatrix} .z^r =\epsilon_1(t)t^r|t|^{1/2+s_1}  z^r 
\]
and $X$ is the space $\mathcal S(\mathbb R)$ of Schwartz functions from $\mathbb{R}$ to $\mathbb{C}$ satisfying 
\[   (\begin{pmatrix} t & \\  & 1 \end{pmatrix}.f)(x)= \epsilon_2(t) |t|^{-1/2+s_2}f(t^{-1}x).
\]


To obtain an analogue of Bernstein-Zelevinsky filtration (c.f. \cite[Proposition 5.12]{BZ76}), we use a Fourier transform to obtain an isomorphism from $\mathcal S(\mathbb{R})$ to $\mathcal S(\mathbb R)$ given by:
\[  \widehat{f}(\eta) =\int_{\mathbb{R}} e^{\sqrt{-1}\eta x}f(x) dx
\]
Let $M_2$ be the mirabolic subgroup in $G_2$. Note that the $M_2$-action on $\mathcal S(\mathbb R)$ induces an action on $\mathcal S(\mathbb R)$ with the $G_1$-action given by:
\[ (t.\widehat{f})(\eta) =|t|^{1/2}\chi_2(t)\widehat{f}(t\eta) .
\]
Now we use Borel's lemma to obtain a filtration on $\mathcal S(\mathbb{R})$, as $M_2$-representations:
\[  0 \rightarrow X_1 \rightarrow X \rightarrow X_2 \rightarrow 0
\]
such that
\begin{enumerate}
\item $X_1$ contains all complex-valued Schwartz functions on $\mathbb{R}$ with all derivatives vanishing at $0$ and $G_1$ acts by a multiplicative translation; and
\item $X_2$ is isomorphic to $\mathbb{C}[[z]]$ such that $t \in \mathbb{R}^{\times}$ acts by
\[  t.z^j=t^{j}|t|^{1/2+s_2} z^j
\]
\end{enumerate}
(Note that $X_1$ is independent of the choice of $\epsilon_2$ and $\chi_2$, which follows by an isomorphism between $\mathcal S(\mathbb{R}^{\times})$ and $\chi \otimes \mathcal S(\mathbb R^{\times})$ given by $f \mapsto (t\mapsto \chi(t)f(t))$.)

In summary, we obtain a $G_1$-filtration on $I(\epsilon,s_1,s_2)$ with successive quotients from top to bottom:
\begin{itemize}
\item for $l\in \mathbb{Z}_{\geq 0}$, $1$-dimensional representations with scalar multiplication by $\epsilon_1(t)t^l|t|^{s_1+\frac{1}{2}}$;
\item for $l \in \mathbb{Z}_{\geq 0}$, $1$-dimensional representations with scalar multiplication by $\epsilon_2(t)t^{l}|t|^{s_2+\frac{1}{2}}$;
\item $X_1$,
\end{itemize}
where $\epsilon_1(t)=\epsilon(\mathrm{diag}(t,1))$ and $\epsilon_2(t)=\epsilon(\mathrm{diag}(1,t))$.

Now we specialize to the case that $s_1=\frac{m-1}{2}$ and $s_2=\frac{-m+1}{2}$ for some positive integer $m$. Then $I(\tau_m, \frac{m-1}{2}, \frac{-m+1}{2})$ has a finite-dimensional module as the quotient. Since the finite-dimensional representation does not admit a Whittaker model, it cannot be a quotient of $X_1$. Hence the finite-dimensional representation occupies the weights $t^{-\frac{m}{2}+1}, \ldots , t^{\frac{m}{2}-1}$. This gives the following:

\begin{proposition} \label{prop filt on discrete series}
The discrete series $\delta(m)$ of $G_2$ admits a $G_1$-filtration whose successive subquotients from top to bottom are:
\begin{itemize}
\item a quotient $Y$, which admits a decreasing filtration $F_l$ with subquotients isomorphic to $2$-dimensional representations of $G_1$ which is the direct sum of the two representations
 \[ t\mapsto |t|^{\frac{m}{2}+l}, \mbox{ and } \quad  t \mapsto \mathrm{sgn}(t)|t|^{\frac{m}{2}+l} \]
 for $l \in \mathbb Z_{\geq 0}$, and moreover, $\varprojlim F_l=Y$.
\item the space $\mathcal S(\mathbb R^{\times})$ of Schwartz functions with $G_1$ acting by $(t.f)(x)=f(xt)$.
\end{itemize}
\end{proposition}

\section{Bruhat filtrations} \label{ss filtrion from parabolic}

\subsection{Bruhat filtrations} \label{ss bruhat filt}
In this section, we describe the Bruhat filtration \cite{CHM00} in our content. There are some extending work \cite{AG08} and \cite{CSu21} while we shall not use it in this article. We are aware of some relevant work in \cite{Ha20}.

Let $T=T_{n+1}$ be the subgroup of $G_{n+1}$ containing all diagonal matrices. Let $B=B_{n+1}$ be the Borel subgroup of $G_{n+1}$ containing all upper triangular matrices. Let $P$ be the parabolic subgroup of $G_{n+1}$ containing matrices $\mathrm{diag}(a, g)$ for $a \in \mathbb{R}^{\times}$ and $g\in G_n$, and all upper triangular matrices.

The mirabolic subgroup $M_{n+1}$ acts on $B \setminus G_{n+1}$ by right multiplication.  Let $\mathcal Q^*$ be the open $M_{n+1}$-orbit
\[ \mathcal Q^* = B \setminus Pw_0M_{n+1}
\]
\begin{align} \label{eqn long element}
 w_0= \begin{pmatrix}  & 1 \\ I_n & \end{pmatrix}
\end{align}
Note that $\mathcal Q^*$ is invariant under $M_{n+1}$-action. Let $\mathcal R=B\setminus PM_{n+1}$ be the complement of $\mathcal Q^*$ in $B \setminus G_{n+1}$.

Let $\chi$ be a character of $B$. Set $\mathcal S= \mathrm{Ind}_B^{G_{n+1}} \chi$, where $\mathrm{Ind}_B^{G_{n+1}}$ is the normalized parabolic induction (see e.g. \cite[Pg 61]{BW00}). Then the discussion in \cite[Pg 166]{CHM00} (also see \cite[1.2.4, 2.2.4]{dCl91}), gives a two-step filtration as $G_n$-representations:
\begin{align} \label{eqn filt on parabolic}
 0 \subset  \mathcal S(\mathcal Q^*)            \subset \mathcal S ,
\end{align}
where $\mathcal S(\mathcal Q^*)$ contains all the functions in $\mathcal S$, all of whose derivatives vanish along $\mathcal R$ \cite{dCl91}, and the quotient will be described below. We may sometimes write $\mathcal S(\mathcal Q^*, \chi)$ if we want to emphasis the group action of $B$ is twisted by a character $\chi$. 

Let $S_n$ be the permutation group on $n$ elements. Let $x_0=1$, and for $k=1,\ldots, n-1$,
\[   x_k=( n+1-k, \ldots, n, n+1) .
\]
Now we further stratify $\mathcal R$ by the $B$-orbits in $B\setminus PM_{n+1}$, where $B$ acts by right translation, which are parametrized by the elements in the set
\[ \widetilde{W}= S_n \sqcup (x_1S_n)\ldots \sqcup (x_{n-1}S_n)
\]
corresponding to the Schubert cells:
\[ C_w := B\setminus B w B ,
\]
for $w \in \widetilde{W}$, where, by abuse of notations, $w$ is the corresponding permutation matrix in $G_{n+1}$.

We shall need the following simple lemma later:
\begin{lemma} \label{lem w non 1}
For any $w \in \widetilde{W}$, $w(n+1)\neq 1$. 
\end{lemma}
\begin{proof}
Write $w=x_iw'$ for some $w' \in S_n$. This follows from $w'(n+1)=n+1$ and $x_i(1)=1$.
\end{proof}

We enumerate elements in $\widetilde{W}$ as $w_1=1,w_2, \ldots, w_{n(n!)}$ such that for $i > j$, $w_i \not < w_j$ i.e. $C_{w_i} \not\subset \overline{C_{w_j}}$. We set a sequence of closed subspaces: for $w \in \widetilde{W}$,
\[   Z_{w_i} = \bigcup_{j\leq i} B \setminus \overline{Bw_jB}=\bigcup_{j\leq i} C_{w_j} .
\]
and $Z^*=B\setminus G_{n+1}$. In particular, $Z^*-Z_{w_{n(n!)}} = \mathcal Q^*$ and 
\[ C_{w_{i+1}}=B\setminus Bw_{i+1}B=Z_{w_{i+1}} - Z_{w_i}
\]
Let 
\[ \mathcal Q_{w_i}=\mathcal Q^* \cup \left(\cup_{j \geq i} \mathcal Q_i \right).
\]

Let $U=U_{n+1}$ (resp. $\bar{U}$) be the unipotent radical (resp. opposite unipotent radical) of $B$ in $G_{n+1}$. Let $\bar{U}^{w^{-1}}=w^{-1}\bar{U}w$ and let $V_w=\bar{U}^{w^{-1}}\cap U$. Now, for each $w \in \widetilde{W}$, let 
\[  D_w =B \setminus Bw\bar{U}^{w^{-1}}.
\]
As shown in \cite[Pg 166]{CHM00}, $D_w$ is a subanalytic neighborhood of $C_w$ in $\mathcal Q_{w}$. Moreover, $D_w$ is diffeomorphic to $\mathbb{R}^{k}$, where $k=\mathrm{dim}~\bar{U}=n(n+1)/2$. We also have the diffeomorphism (see \cite[Proposition 1.1.4.6]{Wa72}):
\begin{align} \label{eqn diffeo}
 \mathbb{R}^{k_1}\times \mathbb{R}^{k_2}\cong  (\bar{U}^{w^{-1}} \cap \bar{U}) \times (\bar{U}^{w^{-1}} \cap U) \rightarrow D_w ,
\end{align}
given by $(u_1, u_2) \mapsto Pwu_1u_2$. Here $k_1=\mathrm{dim}~(\bar{U}^{w^{-1}}\cap \bar{U})$ and $k_2=\mathrm{dim}~(\bar{U}^{w^{-1}} \cap U)$. The closed space $C_w$ is identified with $\left\{ 0 \right\} \times \mathbb{R}^{k_2}$ and $0 \times D_w$ respectively.

 Now we refine the filtration (\ref{eqn filt on parabolic}):
\[   0 \subset \mathcal S(\mathcal Q^*) \subset \mathcal S(\mathcal Q_{w_{n(n!)}}) \subset \ldots \subset \mathcal S(Q_{w_2}) \subset \mathcal S(Q_{w_1})
\] 
For convenience, set $\mathcal S_{n(n!)+1}=\mathcal S(\mathcal Q^*)$, $\mathcal S_i=\mathcal S(\mathcal Q_{w_i})$ for $i=1,\ldots, n(n!)$. Following discussions in \cite{CHM00}, which uses \cite[Lemmas 2.6 and 2.7]{CHM00} and (\ref{eqn diffeo}), for $i=1,\ldots, n(n!)$, set $w =w_i$, 
\begin{align} \label{eqn group isomorphism}
 \mathcal L_i:= \mathcal S_{i}/\mathcal S_{i+1} \cong \mathbb{C}[[x_1, \ldots, x_{k_1}]] \widehat{\otimes} \mathcal S(V_w) ,
\end{align}
where $k_1=n(n+1)/2-l(w_{i})$ and $k_2=l(w_{i})$, and $\widehat{\otimes}$ denotes the completed projective tensor product (see \cite{Tr06}). Let $J_p=J_p(w_i)$ be the ideal in $\mathbb{C}[[x_1, \ldots, x_{k_1}]]$ generated by monomials of degree $p$. Set $F_p=J_p \widehat{\otimes} \mathcal S(V_w)$. This gives that $\mathcal L_i$ admits a decreasing filtration 
\[          \ldots \subset F_1 \subset  F_0 .
\]
Since $\mathcal S(\mathbb{R}^{k_2})$ is nuclear, the completed projective tensor product is exact \cite[Ch. 43 and 50]{Tr06}. Hence,
\[  F_{p-1}/F_p \cong (J_{p-1}/J_{p})\widehat{\otimes} \mathcal S(V_w) .
\]

Let $U'$ be the subgroup of $U$ containing matrices of the form $\begin{pmatrix} I_n & * \\ & 1 \end{pmatrix}$. Note that as matrix multiplications, there is a canonical diffeomorphism (see \cite[Lemma 1.1.4.1]{Wa72}):
\[  (V_w \cap U_n) \times (V_w\cap U') \rightarrow V_w .
\]
We can further decompose $\mathcal S(V_w)$ as 
\[  \mathcal S(V_w) \cong \mathcal S(V_w\cap U_n) \widehat{\otimes} \mathcal S(V_w \cap U')
\]
Now, we further write:
\[  F_{p-1}/F_p \cong J_{p-1}/J_{p} \widehat{\otimes} \mathcal S(V_w\cap U_n)\widehat{\otimes} \mathcal S(V_w\cap U') \otimes \chi^w .
\]
Note that we added the last factor $\chi^w$, so that the above isomorphism is as $T$ and $\mathfrak n$-modules with the following explicit actions. Here $\mathfrak n$ is the Lie algebra of of $U_n$.

\begin{enumerate}
\item We now relabel the variables $x_1, \ldots, x_{k_1}$ such that $J_{p-1}/J_p$ is spanned by the monomials of variables $y_{ab}$ with $a,b \in [1+n]\times [1+n]$ satisfying
\begin{align} \label{eqn greater than cond}
  \quad a>b, \mbox{and } w^{-1}(a)>w^{-1}(b) .
\end{align}
Here $[1+n]=\left\{1,\ldots, n+1 \right\}$. For an element $t=(t_1, \ldots, t_n,t_{n+1})$, the action of $t$ is determined by:
 \begin{align} \label{eqn action on yab}
 t.y_{ab} = t_{w^{-1}(a)}^{-1}t_{w^{-1}(b)}y_{ab} .
\end{align}
\item The action of $T$ on $\mathcal S(V_w\cap U_n)$ and $\mathcal S(V_w\cap U')$ is determined by:
\[ (t.f)(u)= f(t^{-1}ut),
\]
where $u \in V_w \cap U_n$ or $u \in V_w\cap U'$.
\item Let $\mathfrak{x}^w$ be the Lie algebra of $ U^{w^{-1}} \cap U_n$. Note that it is spanned by the elements $\mathfrak e_{w^{-1}(b), w^{-1}(a)}$ with $a,b$ satisfying the condition (\ref{eqn greater than cond}). Here $\mathfrak e_{kl}$ is the $1$-dimensional space spanned by the $(k,l)$ entry. In such case, by taking differentiations, one sees that $\mathfrak x^w$ acts trivially on three factors. 
\item Let $\mathfrak x_{w}$ be the Lie algebra of $V_w =\bar{U}^{w^{-1}}\cap U_{n+1}$. Recall that the $G_n$-group action on an induced representation is by a right translation, and so an element in $\mathfrak x_w$ acts trivially on the factors $J_{p-1}/J_{p}$ and $\mathcal S(V_w \cap U')$. And $\mathfrak x_w$ acts by taking the differentiation on the factor $\mathcal S(V_w \cap U_n)$.
\end{enumerate}

We shall write $|.|$ the determinant character $g \mapsto |\mathrm{det}(g)|$. 

\begin{proposition} \label{prop transfer to fj}
Write $\chi=\chi_1\boxtimes \ldots \boxtimes \chi_{n+1}$ as the tensor product of $(n+1)$ $\mathbb{R}^{\times}$-characters $\chi_1, \ldots, \chi_{n+1}$. Then, as $G_n$-representations,
\[  \mathcal S(\mathcal Q^*, \chi)  \cong |.|^{1/2} \otimes \mathcal S(\mathbb{R}^n) \widehat{\otimes} \mathrm{ind}_{B_n}^{G_n} (\chi_2\boxtimes \ldots \boxtimes \chi_{n+1}). \]
Here $G_n$ acts the factor $\mathcal S(\mathbb R^n)$ by $g.f(v)=f(g^{-1}v)$.
\end{proposition}

\begin{proof}

We observe that
\[ \mathcal Q^* := B_{n+1} \setminus Pw_0B_{n+1} \cong (B_{n}\setminus G_n) \times \mathbb{R}^n ,
\]
where $P$ contains all matrices of the form $\begin{pmatrix} a & * \\ & g \end{pmatrix}$ with $a \in \mathbb{R}^{\times}$ and $g \in G_n$ via the natural map 
\[   B_{n+1} \begin{pmatrix} 1 &  \\ & g \end{pmatrix} w_0 \begin{pmatrix} I_n& v\\ & 1 \end{pmatrix} \mapsto (B_ng , v) .
\]  

The two Schwartz spaces $\mathcal S(\mathcal Q^*)$ and $\mathcal S((B_{n}\setminus G_n) \times \mathbb{R}^n)$ are homeomorphic by \cite[1.2.2. TH\'EOR\`EME, 1.2.6 PROPOSITION]{dCl91} via the above natural map. Now one checks the $B_n$-equivariant properties coincide to descend to a map in the two targeted Schwartz spaces. 
\end{proof}

\subsection{Example of $\mathrm{GL}_3(\mathbb R)$}
Let $\delta_B$ be the modular character of $B=B_3$. Let $\mathfrak n=\mathfrak n_2$ be the Lie algebra of $U_2$ in $G_2$. Let $w_1=1$, $w_2=(1,2)$, $w_3=(2,3)$ and $w_4=(2,3)(1,2)$. Now we can describe a more explicit action on each successive quotients:
\begin{enumerate}
\item $\mathcal S(\mathcal Q_{w_1})/\mathcal S(\mathcal Q_{w_2})$. In this case, 
\[ \mathcal L_1  \cong \mathbb{C}[[y_{21}, y_{31}, y_{32}]] \otimes \chi \delta_B
\]
with $t=\mathrm{diag}(t_1,t_2,1)$ acts by $t.y_{ij}=t_jt_i^{-1}y_{ij}$ for $i>j$, and $\mathfrak n$ acts trivially on the successive subquotients $F_{p-1}/F_p$.
\item $\mathcal S(\mathcal Q_{w_2})/\mathcal S(\mathcal Q_{w_3})$. In this case, 
\[ \mathcal L_2 \cong \mathbb{C}[[y_{31}, y_{32}]] \widehat{\otimes} \mathcal S(\mathbb{R})\otimes (\chi \delta_B)^{w_2}
\]
with $t=\mathrm{diag}(t_1,t_2,1)$ acts by $t.y_{31}=t_2y_{31}$, $t.y_{32}=t_1y_{32}$, 
\[(t.f)(z_{12})=\delta_B\chi(\mathrm{diag}(t_2,t_1,1))f(t_2t_1^{-1}z_{12}) \]
 for $f\in \mathcal S(\mathbb R)$, with $\mathfrak n$ acts by differentiating on the factor $\mathcal S(\mathbb R)$ in each successive subquotients.
\item $\mathcal S(\mathcal Q_{w_3})/\mathcal S(\mathcal Q_{w_4})$. In this case,
\[  \mathcal L_3 \cong \mathbb{C}[[y_{21}, y_{31}]] \widehat{\otimes} \mathcal S(\mathbb R) \otimes (\chi \delta_B)^{w_3}
\] 
with $t=\mathrm{diag}(t_1, t_2,1)$ acts by $t.y_{21}=t_1y_{21}$, $t.y_{31}=t_1t_2^{-1}y_{31}$ and 
\[ t.f(z_{23})=\delta_B\chi(\mathrm{diag}(t_1,1,t_2))f(t_2^{-1}z_{23}) ,\]
and $\mathfrak n$ acts trivially on the successive subquotients.
\item $\mathcal S(\mathcal Q_{w_4})/\mathcal S(\mathcal Q^*)$. In this case,
\[ \mathcal L_4 \cong  \mathbb{C}[[y_{21}]] \widehat{\otimes} \mathcal S(\mathbb{R}^2)\otimes (\chi \delta_B)^{w_4}
\]
with $t=\mathrm{diag}(t_1,t_2,1)$ acts by $t.y_{21}=t_2y_{21}$ and 
\[ (t.f)(z_{12},z_{13})=\delta_B\chi(\mathrm{diag}(t_2,1,t_1))f(t_2t_1^{-1}z_{12},t_1^{-1}z_{13}) ,\]
and $\mathfrak n$ acts by differentiating on  the variable $z_{13}$ in each successive subquotients.
\end{enumerate}

\subsection{$\mathfrak n$-homology} \label{ss cj functor}

When the additive group $\mathbb R$ acts on $\mathcal S(\mathbb R)$ by translation, we have seen in Section \ref{ss filtration}, which uses Fourier transform and Borel's lemma, $\mathcal S(\mathbb R)$ is homeomorphic to $\mathcal S(\mathbb R)$ with the corresponding $\mathbb R$-action given by
\[  (x.f)(\eta)= e^{-\sqrt{-1}\eta x}f(\eta) 
\]
Let $X$ be the corresponding operator by taking differentiation. We have that $(X.f)(\eta)=-\sqrt{-1}\eta f(\eta)$. Denote $\mathfrak n'=\mathbb{C}$, the complexified Lie algebra of $\mathbb{R}$, we shall need the following computation of homology \cite{CHM00}:





\begin{lemma} \cite[Theorem 5.6]{CHM00} \label{lem homology of schwartz}
With the notations above, the Lie algebra homology on $\mathcal S(\mathbb R)$ is given by:
\[  H_i(\mathfrak n',\mathcal S(\mathbb R))\cong \left\{ \begin{array}{cc} \mathbb C & i=0 \\ 0 & i \neq 0 \end{array} \right. 
\]
\end{lemma}

Since $\mathfrak n'$ is abelian, we also have the following well-known result e.g. \cite[Exercise 7.4.2]{We94}:

\begin{lemma}
The Lie algebra homology on the trivial module $\mathbb{C}$ is given by:
\[  H_i(\mathfrak n', \mathbb{C}) =\left\{ \begin{array}{cc} \mathbb{C} & i =0,1 \\ 0 & \mbox{otherwise} \end{array} \right.
\]
\end{lemma}

\subsection{Left-right filtrations}

To apply the idea of left-right filtration in $p$-adic groups, one uses the following result. That is to first obtain a Bruhat filtration on $\theta(\mathrm{ind}_{B_n}^{G_n}\chi)$ and then do $\theta$ to get the filtration on $\mathrm{ind}_{B_n}^{G_n}\chi$, while we shall not state such filtration explicitly. 

\begin{lemma} \label{lem restriction for right}
Let $n_1+n_2=n+1$. Let $\pi_1$ and $\pi_2$ be $G_{n_1}$ and $G_{n_2}$ representations respectively. Then 
\[  (\pi_1 \times \pi_2)|_{G_n} \cong \theta((\theta(\pi_2)\times \theta(\pi_1))|_{G_n}) .
\]
\end{lemma}

\begin{proof}

We firstly have that, for any $\pi_1 \in \mathrm{Alg}(G_{n_1})$ and $\pi_2 \in \mathrm{Alg}(G_{n_2})$, 
\[  \theta(\pi_1)\times \theta(\pi_2) \cong \theta(\pi_2 \times \pi_1) .
\]
It is direct to check that for any $\pi' \in \mathrm{Alg}(G_n)$, 
\[  \theta(\theta(\pi')|_{G_n}) \cong \pi'|_{G_n} .
\]
The result follows by combining the two isomorphisms.
\end{proof}

\begin{example}
We give an example on how the left-right filtration works. Take $\pi=\delta(m)$ to be a discrete series. Then we have that $\theta(\pi) \cong \pi$. On the other hand, $\pi \cong \theta(\theta(\pi)|_{G_1})\cong \theta(\pi|_{G_1})$. Then one obtains a new filtration on $\pi|_{G_1}$ by replacing the representations $|t|^{\frac{m}{2}+l}$ and $\mathrm{sgn}\otimes |t|^{\frac{m}{2}+l}$ ($l\in \mathbb{Z}_{\geq 0}$) in Proposition \ref{prop filt on discrete series} with representations $|t|^{-\frac{m}{2}-l}$ and $\mathrm{sgn}\otimes |t|^{-\frac{m}{2}-l}$ ($l \in \mathbb{Z}_{\geq 0}$).

From here with Lemma \ref{lem homology of schwartz}, indeed one sees that the $\mathfrak{gl}_1$-homology on $H_i(\mathfrak{gl}_1, \delta(m)\otimes \chi)=0$ for any $i \geq 1$ and any $G_1$-character $\chi$ (c.f. \cite{CS21}). 


\end{example}

\section{Multiplicity result: $F=\mathbb R$ } \label{ss archi case multi}

\subsection{Principal standard representations} \label{ss filt standard rep}

Define $\mathrm{sgn}$ to be the unitary character of $G_1$ given by $\mathrm{sgn}(t)=-1$ if $t<0$ and $\mathrm{sgn}(t)=1$ if $t>0$. Define a $G_1$-representation:
\[  \nu(\epsilon, c) = \epsilon \otimes |.|^c ,
\]
where $\epsilon \in \left\{ 1, \mathrm{sgn}\right\}$ and $c \in \mathbb{C}$.

Let 
\[ T = (\mathbb{R}^{\times})^n
\]
Let $\chi=\nu_1 \boxtimes \ldots \boxtimes \nu_n$ be a $T$-character with each $\nu_i=\nu(\epsilon, c_i)$, which extends trivially to a $B_n$-representation. A principal series of $G_{n}$ takes the form:
\[ \nu_1\times \ldots \times \nu_n :=\mathrm{Ind}_{B_n}^{G_n}\nu_1 \boxtimes \ldots \boxtimes \nu_n,
\]
When $\mathrm{Re}(c_1)\geq \ldots \geq \mathrm{Re}(c_n)$,  those principal series are called standard. By the Langlands classification, they have unique irreducible quotients.

\subsection{A lemma}

\begin{lemma} \label{lem vanishing homology}
Suppose there exists a short exact sequence of smooth Fr\'echet representations of $B_n$
\[ 0 \rightarrow Y \rightarrow Z \rightarrow X \rightarrow 0 .
\]
Let $\mathfrak b_n$ be the Lie algebra of $B_n$. Suppose
\[(*)\quad  H_i(\mathfrak b_n, X ) =0
\]
for all $i$. Then, 
\[  \mathrm{Hom}_{B_n} (Y, \mathbb C) \cong \mathrm{Hom}_{B_n}(Z,\mathbb C) ,
\]
where the morphisms are continuous $B_n$-maps.
\end{lemma}

\begin{proof}

Assume (*). Let $\mathfrak b=\mathfrak b_n$. Then a long exact sequence implies that
\[ H_0(\mathfrak b, Z) =H_0(\mathfrak b, Y )
\]
This gives that, as $\mathfrak b$-module morphisms,
\begin{align} \label{eqn isomor}  \mathrm{Hom}_{\mathfrak b}(Z, \mathbb C) \cong \mathrm{Hom}_{\mathfrak b}(Y,\mathbb C) .\end{align}
 Then, for a given map $f \in \mathrm{Hom}_{B_n}(Y, \mathbb C)$, it descends to a $\mathfrak b$-map by taking differentiation. Then the above isomorphism gives a $\mathfrak b$-map $\widetilde{f}$ from $Z$ to $\mathbb{C}$ such that $\widetilde{f}|_Y=f$. We first check that such map is continuous. Let $A=\mathrm{ker} f$, which is closed in $Y$. Now, as $\mathfrak b$-modules, we have a short exact sequence:
\[ 0 \rightarrow        Y / A    \rightarrow Z/A \rightarrow X \rightarrow 0 .
\]
Let $Z'= Z/A$. Using (*), $H_i(\mathfrak b, Y/A)\cong H_i(\mathfrak b,Z')$ for all $i$. Hence, $H_0(\mathfrak b,Z')\cong \mathbb{C}$ since $H_0(\mathfrak b, Y/A)\cong \mathbb{C}$. Then we obtain a unique $\mathfrak b$-map $f'$ from $Z'$ to $\mathbb{C}$, which so must be a descent of $\widetilde{f}$ (up to some non-zero scalar). Now, using (*) again, $\ker ~f' =\mathfrak bZ$. The natural map $\mathfrak{b} \otimes Z\mapsto Z$ is continuous with finite codimension image. Hence $\ker ~f'$ is closed in $Z $ by \cite[Lemma A.1]{CHM00} and hence $f'$ is continuous. Hence $\widetilde{f}=f' \circ \mathrm{pr} $ is also continuous, where $\mathrm{pr}$ is the projection map. Now $\mathrm{ker}~\widetilde{f}$ contains the closure of $\mathfrak bZ$ and so $Z/\mathrm{ker}~\widetilde{f}$ is a smooth Fr\'echet space with trivial $\mathfrak b$-action. Hence $\widetilde{f}$ is also a $B_n$-map with $B_n$ acting trivially. Hence we have a surjection from $\mathrm{Hom}_{B_n}(Z, \mathbb C)$ to $\mathrm{Hom}_{B_n}(Y,\mathbb C)$. Since $\mathrm{Hom}_{B_n}(X,\mathbb C)=0$ by (*), the map is also an injection as desired.

\end{proof}

\subsection{Equal rank Fourier-Jacobi models}
For an admissible representation $\pi$ of $G_n$, let $\pi^{\vee}$ be the smooth dual of $\pi$ as in \cite{BK14} i.e. the Casselman-Wallach globalization of the algebraic dual of the Harish-Chandra counterpart of $\pi$. The proof strategy of the following theorem has similar spirit as the non-Archimedean case and we focus on the technique of left-right filtrations (Lemma \ref{lem restriction for right}, c.f. \cite[Section 5]{Ch21}) and a relation between equal rank Fourier-Jacobi model and Bessel models (Proposition \ref{prop transfer to fj}, c.f. \cite[Section 5.1]{Ch20}, \cite[Section 8.3]{Ja09}). 

\begin{theorem} \label{thm eq rank FJ}
Let $\pi_1$ and $\pi_2$ be standard principal series of $G_{n}$. Then 
\[  \mathrm{dim}~ \mathrm{Hom}_{G_n}(\pi_1 \widehat{\otimes} \mathcal{S}(\mathbb R^n)\otimes |.|^{1/2}, \pi_2^{\vee}) = 1.
\]
Here $G_n$ acts on the factor $\mathcal S(\mathbb R^n)$ by $(g.f)(v)=f(g^{-1}v)$.
\end{theorem}

\begin{remark}
The normalization $|.|^{1/2}$ is not essential since $\pi_1\otimes |.|^{1/2}$ is still a standard principal series, but it relates to the branching problem from $G_{n+1}$ to $G_n$ (c.f. \cite{Ch20}). 
\end{remark}

\begin{proof}
Write $\pi_1=\mathrm{ind}_{B_{n}}^{G_{n}}\chi_1$ and $\pi_2=\mathrm{ind}_{B_n}^{G_n}\chi_2$. 
For $i=1,\ldots, n+1$, write $\nu_i=\nu(\epsilon_i,c_i)$, and for $i=1,\ldots, n$, write $\nu_i'=\nu(\epsilon_i',c_i')$. Then 
\[ \chi_1= \nu(\epsilon_1,c_1) \boxtimes \ldots \boxtimes \nu(\delta_{n},c_{n})
\]
and 
\[  \chi_2=\nu(\epsilon_1', c_1')\boxtimes \ldots \boxtimes \nu(\epsilon_n', c_n') .
\]
Two numbers $a,b$ in $\mathbb C$ are said to be {\it linked} if $a-b\in \mathbb{Z}$. Let $N^*$ be the sum of three terms:
\begin{enumerate}
\item the number of linked pairs $(c_i,c_j)$ ($i<j$);
\item the number of linked pairs $( c_i',c_j' )$ ($i<j$);
\item the number of linked pairs $(c_i+\frac{1}{2}, -c_j')$ for any $i,j$.
\end{enumerate}
 We shall prove by an induction on $N^*$. When $N^*$ is zero, $\pi_1$ and $\pi_2$ are irreducible by e.g. \cite{Sp81}. Then the required statement follows from \cite{LS13} (also see Remark \ref{rmk deduce equal rank case}). 

We pick the smallest $i^*$ such that $c_{i^*}$ is linked to some $c_j$ ($j>i^*$) or linked to some $-c_j'-\frac{1}{2}$. Then we will consider the following three cases:
\begin{enumerate}
\item There exists a linked pair $(c_i,c_j)$ $(i<j)$;
\item There exists a linked pair $(c_i',c_j')$ $(i<j)$;
\item There exists a linked pair $(c_i+\frac{1}{2}, -c_j')$.
\end{enumerate}
We shall only give a proof for the third case and the other two cases can be proved by some similar arguments.

We now consider the third case. Let $i^*$ be the smallest integer such that $(c_{i^*}+\frac{1}{2}, -c_{j}')$ is linked for some $j$. One observes that for any $i<i^*$, $(c_i,c_{i^*})$ is not linked. Let $j^*$ be the largest integer such that $(c_{i^*}+\frac{1}{2}, -c'_{j^*})$ is linked. Again, one observes that for any $j>j^*$, $(c'_{j^*},c'_j)$ is not linked. \\

\noindent
{\bf Case 1.} $\mathrm{Re}(c_{i^*}+\frac{1}{2})> \mathrm{Re}(-c_{j^*}')$.

Set $\widetilde{\chi}=\nu^*\boxtimes\nu_1 \boxtimes \ldots \boxtimes \nu_n \in T_{n+1}$, where $\nu^*=\nu(\epsilon^*, c^*)$ such that $(c^*,c_i)$ and $(c^*+\frac{1}{2},-c_j')$ are not linked for any $i,j$.

Let $\widetilde{\pi}=\mathrm{Ind}_{B_{n+1}}^{G_{n+1}}\widetilde{\chi}$. The idea of passing to a larger representation is to exploit some irreducibility result on parabolic inductions. We denote $\mathcal S=\widetilde{\pi}$ as in Section \ref{ss bruhat filt}. We have an exact sequence:
\[  0 \rightarrow  \mathcal S(\mathcal Q^*, \widetilde{\chi})\otimes (\chi_2\delta_{B_n}^{-1})  \rightarrow  \widetilde{\pi}\otimes (\chi_2\delta_{B_n}^{-1}) \rightarrow  (\mathcal S/\mathcal S(\mathcal Q^*, \widetilde{\chi})) \otimes (\chi_2\delta_{B_n}^{-1}) \rightarrow 0
\]
By Proposition \ref{prop transfer to fj}, the exact sequence can be rewritten as:
\[  0 \rightarrow  \pi_1\widehat{\otimes}\mathcal S(\mathbb R^n)\otimes (\chi_2\delta_{B_n}^{-1}) \rightarrow  \widetilde{\pi}\otimes (\chi_2\delta_{B_n}^{-1}) \rightarrow  (\mathcal S/\mathcal S(\mathcal Q^*, \widetilde{\chi})) \otimes (\chi_2\delta_{B_n}^{-1}) \rightarrow 0
\]
Denote the last term by $X$, and the first term by $Y$. \\

We begin with proving the following lemma. Let $\mathfrak b_n$ be the complexified Lie algebra of $B_n$. We write $\mathfrak b_n=\mathfrak t_n+\mathfrak n_n$, where $\mathfrak t_n$ and $\mathfrak n_n$ are the complexified Lie algebras of $T$ and $U_n$ respectively. We shall sometimes simply write as $\mathfrak b$, $\mathfrak t$ and $\mathfrak n$.

\begin{lemma} \label{lem transfer lemma}
 $\mathrm{Hom}_{G_n}(\widetilde{\pi}, \pi_2^{\vee}) \cong \mathrm{Hom}_{G_n}(\pi_1\widehat{\otimes} \mathcal S(\mathbb R^n) \otimes |.|^{1/2}, \pi_2^{\vee})$. 
\end{lemma}
\begin{proof}
 By Frobenius reciprocity, it is equivalent to show
\[  \mathrm{Hom}_{B_n}(\widetilde{\pi}, \chi_2^{\vee}\delta_{B_n}) \cong \mathrm{Hom}_{B_n}(\pi_1\widehat{\otimes}\mathcal S(\mathbb R^n), \chi_2^{\vee}\delta_{B_n}) .
\]

By Lemma \ref{lem vanishing homology}, it suffices to check (*). We shall proceed by using the Bruhat filtration in Section \ref{ss bruhat filt} and freely use the notations in Section \ref{ss bruhat filt}.





 We shall show 
\[  H_i(\mathfrak b, (\mathcal S(\mathcal Q_{w_j})/\mathcal S(\mathcal Q_{w_{j+1}}))\otimes \chi_2\delta_{B_n}^{-1})=0 \]
for each $w_j$, which will follow (see e.g. \cite[Theorem 3.5.8]{We94}) if we could show that
\[  H_i(\mathfrak b, F_p/F_{p+1})=0 ,
\]
where $F_p/F_{p+1} \cong J_p/J_{p+1}\otimes \mathcal S(V_{w_j}\cap U_n) \widehat{\otimes} \mathcal S(V_{w_j}\cap U') \otimes (\chi_1\chi_2\delta_{B_{n+1}}\delta_{B_n}^{-1})$.

Set $w=w_j$. Now we can further write
\[  F_p/F_{p+1} \cong J_p/J_{p+1}\otimes \mathcal S(V_w\cap U_n) \widehat{\otimes} \mathcal S(V_w\cap U') \otimes \chi',
\]
where $\chi'=\chi_1\chi_2\delta_{B_{n+1}}\delta_{B_n}^{-1}$. As we discussed in Section \ref{ss bruhat filt}, $\mathfrak n$ acts trivially on all factors except $\mathcal S(V_w\cap U_n)$. 
\[  H_i(\mathfrak n,F_p/F_{p+1})=A\otimes H_i(\mathfrak n, \mathcal S(V_w\cap U_n)) ,
\]
\[ A =J_p/J_{p+1}\otimes \mathcal S(V_w\cap U') \otimes \chi' .
\] 
Now one may follow computations in \cite{CHM00}. One finds a sequence of subalgebras $\mathfrak m_1, \ldots, \mathfrak m_p$ of $\mathfrak n$ such that the successive subquotients $\mathfrak m_{i+1}/\mathfrak m_{i}$ is abelian. Then one applies spectral sequence arguments on $H_r(\mathfrak m_{i+1}/\mathfrak m_i,H_s(\mathfrak m_i/\mathfrak m_{i-1},.))$ (for $i=1,\ldots ,p-1$) to get that $H_r(\mathfrak n, \mathcal S(V_w\cap U_n))$ is isomorphic, as $\mathfrak t$-module, a sum of some $\mathfrak t$-subspaces of $\mathfrak x^w$ (see discussions in Section \ref{ss cj functor}).

By a spectral sequence argument again, it suffices to show that, for all $r,s$,
\[  H_r(\mathfrak t,  H_s(\mathfrak n, F_p/F_{p+1} )) =0
\]
To this end, let $k^*= w_i^{-1}(1)$. Note that $k^*\neq n+1$ by Lemma \ref{lem w non 1}. Let $h=\mathrm{diag}(0,\ldots, 0,1,0,\ldots, 0) \in \mathfrak t$ be the element with $1$ in $(k^*,k^*)$-entry. Recall that $J_p/J_{p+1}$ is spanned by the monomials $y_{rs}$ with $r>s$, and those that can contribute a weight to $h$ are those $y_{r1}$, and so in such case, $h$ acts with a non-negative integral weight on $J_p/J_{p+1}$. 

Moreover, as discussed above, $H_r(\mathfrak n, \mathcal S(V_w\cap U_n))$ is isomorphic to finite sum of some $\mathfrak t$-subspaces of $\mathfrak x^w$ (see Section \ref{ss bruhat filt} for a description of $\mathfrak x^w$). Now the condition (\ref{eqn greater than cond}) gives that the $1$-dimensional $\mathfrak t$-subspaces of $H_r(\mathfrak n, \mathcal S(V_w \cap U_n))$ that can contribute to a non-zero integral weight of $h$ is isomorphic to $\mathfrak e_{a,k^*}$ with $a<k^*$, where $\mathfrak e_{a,k^*}$ is the space of matrices with zero outside the entry $(a,k^*)$. Since $h$ must have a positive integral weight on $\mathfrak e_{a,k^*}$, $h$ has a non-negative integral weight on $H_r(\mathfrak n, F_p/F_{p+1})$. 

We now check that $h$ acts trivially on the part $\mathcal S(V_w\cap U')$. But the condition $w^{-1}(a)<w^{-1}(b)$ for $a>b$ gives non-zero entries $(w^{-1}(a),w^{-1}(b))$ in $V_w \cap U_{n+1}$, but then we cannot have $a=1$ and so no non-zero $(k^*,.)$ entries in $V_w \cap U_{n+1}$, particularly $(k^*,n+1)$-entry. Hence, $h$ also acts trivially on $\mathcal S(V_w\cap U')$.

Now combining with $w(h)$ acting by a non-integral weight on $\chi_1\chi_2\delta_{B_{n+1}}\delta_{B_n}^{-1}=\chi_1\chi_2 |.|^{1/2}$ (by using non-linkedness), all weights of $h$ on $F_p/F_{p+1} $ is non-zero. Thus
\[  H_r(\mathfrak t, H_s(\mathfrak n, F_p/F_{p+1})) =0
\]
for any $r,s$. Now a Hochschild-Serre spectral sequence argument gives that 
\[  H_r(\mathfrak b, F_p/F_{p+1})=0 
\]
for all $r$.
\end{proof}

We continue the proof of Theorem \ref{thm eq rank FJ}. By the irreducibility of representations of $\mathrm{GL}_2(\mathbb R)$ (see e.g. \cite{Sp81}) and \cite[Proposition 4.1.12]{Vo81},  the paragraph before Case 1 gives that 
\[  \widetilde{\pi} \cong \nu_{i^*} \times \nu^* \times \nu_1 \times\ldots \times \nu_{i^*-1}\times \nu_{i^*+1}\times \ldots \times \nu_{n+1} .
\]
Let $\pi_1'=\nu^* \times \nu_1\times \ldots \times \nu_{i^*-1} \times \nu_{i^*+1}\times \ldots \times \nu_{n+1}$. Note that Lemma \ref{lem transfer lemma} and its proof still work if one replaces $\pi_1$ by $\pi_1'$, in which one has to replace the linked condition in the last paragraph of the proof of Lemma \ref{lem transfer lemma} by the condition $\mathrm{Re}(c_{i^*}+\frac{1}{2})> \mathrm{Re}(-c_{j^*}')$. In other words, we have that:
\[\mathrm{Hom}_{G_n}(\pi_1\widehat{\otimes}\mathcal S(\mathbb R^n)\otimes |.|^{1/2}, \pi_2^{\vee})\cong  \mathrm{Hom}_{G_n}(\widetilde{\pi}, \pi_2^{\vee}) \cong \mathrm{Hom}_{G_n}(\pi_1'\widehat{\otimes} \mathcal S(\mathbb R^n) \otimes |.|^{1/2}, \pi_2^{\vee}) .
\]
By induction, the last space has dimension one, and so does the first one. \\

{\bf Case 2.} $\mathrm{Re}(c_{i^*}+\frac{1}{2}) \leq \mathrm{Re}(-c_{j^*}')$. Recall that $(c_{i^*}+\frac{1}{2}, -c_{j^*}')$ is linked, and so
\[  \mathrm{Re}(-c_{j^*}'+\frac{1}{2}) > \mathrm{Re}(c_{i^*}) .
\]
(Note that here is a strict inequaltiy.) 

Now one has that
\begin{align}
  \mathrm{Hom}_{G_n}(\pi_1 \widehat{\otimes}\mathcal S(\mathbb{R}^n)\otimes |.|^{1/2}, \pi_2^{\vee})  &\cong \mathrm{Hom}_{G_n}(\pi_1\widehat{\otimes}\pi_2\widehat{\otimes}\mathcal S(\mathbb{R}^n)\otimes |.|^{1/2}, \mathbb{C}) \\
	& \cong \mathrm{Hom}_{G_n}(\pi_2\widehat{\otimes} \mathcal S(\mathbb{R}^n) \otimes |.|^{1/2}, \pi_1^{\vee}) ,
\end{align}
where the two isomorphisms follow from taking the contragredient representation and the uniqueness of Casselman-Wallach globalization.

Similar to Case 1, we consider 
\[  \widehat{\pi}=\nu(\widehat{\epsilon}, \widehat{c}) \times \nu(\epsilon_1',c_1')\times \ldots \times \nu(\epsilon_n', c_n') ,\]
where $\widehat{c}$ satisfies that all pairs $(\widehat{c}, c_j')$ and $(\widehat{c}-\frac{1}{2}, -c_j)$ are not linked. The proof of Lemma \ref{lem transfer lemma} still applies to get that:
\begin{align} \label{eqn descent technique}
   \mathrm{Hom}_{G_n}(\widehat{\pi}, \pi_1^{\vee}) \cong \mathrm{Hom}_{G_n}(\pi_2 \widehat{\otimes} \mathcal S(\mathbb R^n)\otimes |.|^{1/2}, \pi_1^{\vee}) .
\end{align}

Now one applies $\theta$, we have that 
\begin{align}
\mathrm{Hom}_{G_n}(\widehat{\pi}, \pi_1^{\vee}) & \cong   \mathrm{Hom}_{G_n}(\theta(\widehat{\pi}|_{G_n}), \theta(\pi_1^{\vee}))  \\
 & \cong \mathrm{Hom}_{G_n}(\theta(\widehat{\pi}|_{G_n}), \theta(\pi_1)^{\vee}).
\end{align}

Now, by an irreducibility criteria again \cite{Sp81} and \cite{Vo81}, 
\[  \widehat{\pi} \cong \widehat{\pi}'\times \nu(\epsilon_{j^*}, c_{j^*}) ,
\]
where 
\[ \widehat{\pi}'= \nu(\widehat{\epsilon}, \widehat{c}) \times \nu(\epsilon_1',c_1')\times \ldots \times  \nu(\epsilon_{j^*-1}',c_{j^*-1}')\times \nu(\epsilon_{j^*+1}',c_{j^*+1}') \times \ldots \times \nu(\epsilon_n', c_n') .\]
So now, as in Lemma \ref{lem restriction for right}, 
\[ \theta(\widehat{\pi}|_{G_n}) \cong \nu(\epsilon_{j^*},-c_{j^*}')\times \theta(\widehat{\pi}')
\]
With such realization, the situation is similar to Lemma \ref{lem transfer lemma}  again and so the proof still applies to obtain that 
\[  \mathrm{Hom}_{G_n}(\theta(\widehat{\pi}')\widehat{\otimes}\mathcal{S}(\mathbb{R}^n)\otimes |.|^{1/2}, \theta(\pi_1)^{\vee}) \cong \mathrm{Hom}_{G_n}(\theta(\widehat{\pi}|_{G_n}), \theta(\pi_1)^{\vee})
\]
The last Hom has dimension equal to one by induction. Now tracing all the isomorphisms, we have that 
\[ \mathrm{dim}~\mathrm{Hom}_{G_n}(\pi_1\widehat{\otimes}\mathcal S(\mathbb R^n)\otimes |.|^{1/2}, \pi_2^{\vee}) = 1
\]
as desired.
\end{proof}

\begin{remark} \label{rmk deduce equal rank case}
\begin{itemize}
\item We emphasis that the analogous statement of Theorem \ref{thm eq rank FJ} for irreducible representations is obtained by Liu-Sun \cite{LS13}.
\item Instead of using \cite{LS13}, one may also use the above argument (in particular, Lemma \ref{lem transfer lemma}) to deduce the multiplicity one for irreducible principle series of equal rank case from the corank one case in \cite{SZ12}. 
\end{itemize}
\end{remark}

\section{Multiplicity result: $F=\mathbb{C}$} \label{ss multi complex}

In this section, we deal with $F=\mathbb{C}$. We shall only comment on the modification that is needed from $F=\mathbb R$. For $r, s \in \mathbb{C}$ satisfying $r-s \in \mathbb{Z}$, define a character $\nu_{r,s}: \mathbb{C}^{\times}\rightarrow \mathbb{C}^{\times}$ given by
\[  \nu_{r,s}(z) =z^r\bar{z}^s .
\]

The only essentially discrete series appears in $\mathrm{GL}_1(\mathbb C)$. Let $\nu_{r_1, s_1}, \ldots, \nu_{r_n, s_n}$ be characters of $\mathrm{GL}_1(\mathbb C)$. We again set $G_n=\mathrm{GL}_n(\mathbb C)$. Any standard modules take the form of a normalized parabolically induced representation:
\[ \nu_{r_1,s_1}\times \ldots \times \nu_{r_n,s_n}:=\mathrm{Ind}_B^{G_n} \nu_{r_1,s_1}\boxtimes \ldots \boxtimes \nu_{r_n,s_n}
\]
with $\mathrm{Re}(r_1+s_1)\geq \ldots \geq \mathrm{Re}(r_n+s_n)$. Here $B$ is the Borel subgroup containing all upper triangular matrices. In particular, all standard modules of $\mathrm{GL}_n(\mathbb C)$ are principal series.

For irreducibility, we have that for two characters $\nu_1, \nu_2$ of $G_1$, $\nu_1 \times \nu_2$ is reducible if and only if there exist integers $p,q \in \mathbb{Z}$ with $pq >0$ such that $\nu_1\nu_2^{-1}(z)=z^p\bar{z}^q$ (see e.g. \cite{Go70}). 

Treatments for $F=\mathbb{R}$ can still apply for $F=\mathbb C$. 
Let $|t|_{\mathbb C}=t\bar{t}$.  For example, for a $G_1 \times G_1$-character $\chi=\nu_1\boxtimes \nu_2$, the space $\mathrm{ind}_{B_2}^{G_2}\chi$ admits a short exact sequence as $G_1$-representations:
\[ 0 \rightarrow \mathcal S(\mathbb{C}) \rightarrow \mathrm{ind}_{B_2}^{G_2}\chi \rightarrow \mathbb{C}[[z,\bar{z}]]  \rightarrow 0 , \]
where $\mathcal S(\mathbb{C}) \cong \mathcal S(\mathbb{R}^2)$ with $\mathbb{C}^{\times}$-action given by:
\[ (t.f)(z)=|t|_{\mathbb C}^{-1/2}\nu_2(t)f(zt^{-1}) ,\]
 and $\mathbb{C}[[z,\bar{z}]]$ with $\mathbb{C}^{\times}$-action by
\[ t.z=|t|_{\mathbb{C}}^{1/2}\nu_1(t)tz, \quad t.\bar{z}=|t|_{\mathbb C}^{1/2}\nu_1(t)\bar{t}\bar{z} .\]

We also have an analogue of Lemma \ref{lem homology of schwartz}: $H_i(\mathfrak n', \mathcal S(\mathbb{C}))$ $\cong \mathbb{C}$ when $i=0$ and vanishes when $i \neq 0$. Here $\mathfrak n'$ is the complexified Lie algebra of $\mathbb C$.

The $G_n$-action on $\mathcal S(\mathbb C^n)$ is given by $(g.f)(v)=f(g^{-1}v)$. We omit the details of the proof of the following result:

\begin{theorem}
Let $\pi_1$ and $\pi_2$ be standard representations of $\mathrm{GL}_n(\mathbb{C})$. Then 
\[  \mathrm{dim}~\mathrm{Hom}_{\mathrm{GL}_n(\mathbb{C})}(\pi_1\widehat{\otimes}\mathcal S(\mathbb{C}^n) \otimes |\mathrm{det}(.)|_{\mathbb C}^{1/2}, \pi_2^{\vee}) = 1 .
\]
\end{theorem}





\end{document}